\documentclass[12pt,twoside,reqno]{amsart}
\usepackage{amssymb}

\usepackage[T2A]{fontenc}
\usepackage[cp1251]{inputenc}
\usepackage[english]{babel}
\usepackage{color}
 \RequirePackage{srcltx}
\advance\textwidth35mm \advance\hoffset-15mm
\advance\textheight40mm \advance\voffset-15mm
\theoremstyle{plain}
\newtheorem{thm}{Theorem}[section]

\newtheorem{lem}[thm]{Lemma}
\newtheorem*{probA}{Problem~A}
\newtheorem*{probB}{Problem~B}
\newtheorem*{probC}{Problem~C}
\newtheorem*{probD}{Problem~D}
\newtheorem*{probE}{Problem~E}
\theoremstyle{remark}
\newtheorem{rem}[thm]{Remark}

\numberwithin{equation}{section}

\def\ch{\operatorname{cosh}}
\def\sh{\operatorname{sinh}}
\def\arcch{\operatorname{arcosh}}

\begin{document}

\title{Logan's problem for Jacobi transforms}

\author{D.~V.~Gorbachev}
\address{D.\,V.~Gorbachev, Tula State University,
Department of Applied Mathematics and Computer Science,
300012 Tula, Russia}
\email{dvgmail@mail.ru}

\author{V.~I.~Ivanov}
\address{V.\,I.~Ivanov, Tula State University,
Department of Applied Mathematics and Computer Science,
300012 Tula, Russia}
\email{ivaleryi@mail.ru}

\author{S.~Yu.~Tikhonov}
\address{S.\,Yu.~Tikhonov,
Centre de Recerca Matem\`atica, Campus de Bellaterra, Edifici C 08193
Bellaterra, Barcelona, Spain; ICREA, Pg. Llu\'is Companys 23, 08010 Barcelona,
Spain, and Universitat Aut\'onoma de Barcelona}
\email{stikhonov@crm.cat}

\date{\today}



\thanks{The work of the first and second authors was supported by RScF,
grant 18-11-00199, https://rscf.ru/project/18-11-00199/. The work of the third
author was partially supported by grants PID2020-114948GB-I00, 2017 SGR 358,
AP08856479, by the CERCA Programme of the Generalitat de Catalunya, and by the
Spanish State Research Agency, through the Severo Ochoa and Mar{\'\i}a de
Maeztu Program for Centers and Units of Excellence in R$\&$D
(CEX2020-001084-M)}

\maketitle

\begin{abstract}
We consider direct and inverse Jacobi transforms with measures
$d\mu(t)=2^{2\rho}(\sh t)^{2\alpha+1}(\ch t)^{2\beta+1}\,dt$ and
$d\sigma(\lambda)=(2\pi)^{-1}\bigl|\frac{2^{\rho-i\lambda}\Gamma(\alpha+1)\Gamma(i\lambda)}
{\Gamma((\rho+i\lambda)/2)\Gamma((\rho+i\lambda)/2-\beta)}\bigr|^{-2}\,d\lambda$,
respectively. We solve the following generalized Logan problem: to find
\[
\inf\Lambda((-1)^{m-1}f), \quad m\in \mathbb{N},
\]
where $\Lambda(f)=\sup\,\{\lambda>0\colon f(\lambda)>0\}$ and the infimum is
taken over all nontrivial even entire functions $f$ of exponential type that
are Jacobi transforms of positive measures with supports on an interval.
Here, if $m\ge 2$, then we additionally assume that
$\int_{0}^{\infty}\lambda^{2k}f(\lambda)\,d\sigma(\lambda)=0$ for
$k=0,\dots,m-2$.

We prove that admissible functions for this problem are positive definite with
respect to the inverse Jacobi transform. The solution of Logan's problem was
known only when $\alpha=\beta=-1/2$. We find a  unique (up to multiplication by
a positive constant) extremizer $f_m$. The corresponding Logan problem for the
Fourier transform on the hyperboloid $\mathbb{H}^{d}$ is also solved.
Using properties of the extremizer $f_m$ allows us to  give an upper estimate of
the length of a minimal interval containing
not less than  $n$ zeros of
positive definite functions.
Finally, we show that the Jacobi functions form the Chebyshev systems.
\end{abstract}

\section{Introduction}

In this paper we continue the discussion of the generalized Logan problem
for 
 entire functions of exponential type, that are,
functions represented as
compactly supported integral transforms.
 In \cite{GIT19}, we investigated this problem for the Fourier, Hankel, and Dunkl transforms. Here we consider the new case of the Jacobi transform, which is closely related to the harmonic analysis on
  the real hyperbolic spaces \cite{stri}.

 The one-dimensional Logan problem first appeared as a problem in the number theory (see
\cite{Lo83a}). Its multidimensional analogues  are also connected to the number-theoretical methods including
the Selberg sieve
 \cite{Va85}.
Considering various classes of admissible functions in the multivariate Logan's problem  gives rise to
the so called  Delsarte extremal problems, which have numerous applications to
discrete mathematics and metric geometry, see the discussion in \cite{GIT19}.
At present, Logan's and Delsarte's  problems can be considered as an important  part
of uncertainty type extremal problems, where conditions both on a function and its Fourier transform are imposed
\cite{GOS17,CohGon18,GonOliRam21} (see also \cite{Be99,KolRev03,CarMilSou19}).

Typically, the main object in such questions   is the classical Fourier transform in  Euclidean space.
However, similar questions in other  symmetric spaces, especially in hyperbolic spaces,
are of great interest, see e.g. \cite{CZ14,GorIvaSmi17}.
Harmonic analysis in these cases is built with the help of the Fourier--Jacobi transform, see
 \cite{Koor84,stri}.
A particular case of the Jacobi transform is the well-known Mehler--Fock transform \cite{Vil78}.

We would like to stress
  that since
the classes of admissible functions and  corresponding functionals in the multidimensional Logan's problem (see Problem~E for hyperboloid in Section~\ref{sec-Hyp}) are invariant under the group of motions,
the problem reduces to the case of radial functions. Thus, it is convenient to start with the one-dimensional case (Problem~D) and then find a solution
 in the whole generality.

  Positive definiteness of the extremizer in the generalized Logan problem for the Hankel transform (see Problem C below)
turns out to be crucial
to obtain lower bounds for energy in the Gaussian core
model
 \cite{CC18}.
We will see that the extremizer in  the generalized Logan problem for the Jacobi
transform is also positive definite (with respect to Jacobi transform).

\subsection*{Historical background 
}
Logan stated and proved \cite{Lo83a,Lo83b} the following two extremal problems for real-valued
positive definite bandlimited functions on $\mathbb{R}$. Since such functions are even,
we consider these problems for functions on $\mathbb{R}_{+}:=[0,\infty)$.

\begin{probA}
Find the smallest $\lambda_1>0$ such that
\[
 f(\lambda)\le 0,\quad \lambda>\lambda_1,
\]
where $f$ is entire function of exponential type at most $2\tau$ satisfying
\begin{equation}\label{eq1.1}
f(\lambda)=\int_{0}^{2\tau}\cos\lambda t\,d\nu(t),\quad f(0)=1,
\end{equation}
where $\nu$ is a function of bounded variation, non-decreasing in some neighborhood of the origin.
\end{probA}

Logan showed that admissible functions are integrable, $\lambda_1=\pi/2\tau$, and the unique
extremizer is the positive definite function
\[
f_1(\lambda)=\frac{\cos^2(\tau\lambda)}{1-\lambda^2/(\pi/2\tau)^2},
\]
satisfying $\int_{0}^{\infty}f_1(\lambda)\,d\lambda=0$.

Recall that a function $f$ defined on $\mathbb{R}$ is positive definite if for any integer $N$
\[
\sum_{i,j=1}^Nc_i\overline{c_j}\,f(x_i-x_j)\ge 0,\quad \forall\,c_1,\dots,c_N\in
\mathbb{C},\quad \forall\,x_1,\dots,x_N\in \mathbb{R}.
\]
Let $C_b(\mathbb{R}_{+})$ be the space of continuous bounded
functions $f$ on $\mathbb{R}_{+}$ with norm
$\|f\|_{\infty}=\sup_{\mathbb{R}_{+}}|f|$.
For an even function $f\in C_b(\mathbb{R}_{+})$, by Bochner's theorem, $f$ is positive definite if and only~if
\begin{equation*}
f(x)=\int_{0}^{\infty}\cos\lambda t\,d\nu(t),
\end{equation*}
where $\nu$ is a non decreasing function of bounded variation (see, e.g.,
\cite[9.2.8]{Ed79}). In particular, if $f\in L^{1}(\mathbb{R}_{+})$, then
its cosine Fourier transform is nonnegative.

\begin{probB}
Find the smallest $\lambda_2>0$ such that
\[
f(\lambda)\ge 0,\quad \lambda>\lambda_2,
\]
where $f$ is an integrable function satisfying \eqref{eq1.1} and having mean value zero.
\end{probB}

It turns out that admissible functions are integrable with respect to the weight $\lambda^2$,
and $\lambda_2=3\pi/2\tau$. Moreover, the unique extremizer is the positive definite function
\[
f_2(\lambda)=\frac{\cos^2(\tau\lambda)}{(1-\lambda^2/(\pi/2\tau)^2)(1-\lambda^2/(3\pi/2\tau)^2)},
\]
satisfying $
\int_{0}^{\infty}\lambda^{2}f_2(\lambda)\,d\lambda=0$.

Let $m\in\mathbb{N}$. Problems A and B can be considered as special cases of the generalized $m$-Logan problem.

\begin{probC}
Find the smallest $\lambda_m>0$ such that
\[
(-1)^{m-1}f(\lambda)\le 0,\quad \lambda>\lambda_m,
\]
where, for $m=1$, $f$ satisfies \eqref{eq1.1} and, for $m\ge 2$, additionally
\[
f\in L^1(\mathbb{R}_{+}, \lambda^{2m-4}\,d\lambda),\quad \int_{0}^{\infty}\lambda^{2k}
f(\lambda)\,d\lambda=0,\quad k=0,\dots,m-2.
\]
\end{probC}

If $m=1,2$ we recover problems A and B, respectively. We solved  problem~C in \cite{GIT19}.
It turns out that the unique extremizer is the positive definite function
\[
f_m(\lambda)=\frac{\cos^2(\tau\lambda)}{(1-\lambda^2/(\pi/2\tau)^2)
(1-\lambda^2/(3\pi/2\tau)^2)\cdots(1-\lambda^2/((2m-1)\pi/2\tau)^2)},
\]
satisfying $f_m\in L^1(\mathbb{R}_{+}, \lambda^{2m-2}\,d\lambda)$ and
$\int_{0}^{\infty}\lambda^{2m-2}f_m(\lambda)\,d\lambda=0$.

Moreover, we have solved a more general version of the problem C when
$f$ is the Hankel transform of a measure, that is,
\[
f(\lambda)=\int_{0}^{2\tau}j_{\alpha}(\lambda t)\,d\nu(\lambda),\quad f(0)=1,
\]
where $\nu$ is a function of bounded variation, non-decreasing in some neighborhood of the origin,
and for $m\ge 2$
\[
f\in L^1(\mathbb{R}_{+}, \lambda^{2m+2\alpha-3}\,d\lambda),\quad \int_{0}^{\infty}
\lambda^{2k+2\alpha+1}f(\lambda)\,d\lambda=0,\quad k=0,\dots,m-2.
\]
Here $\alpha\ge -1/2$ and $j_{\alpha}(t)=(2/t)^{\alpha}\Gamma(\alpha+1)J_{\alpha}(t)$
is the normalized Bessel function. The representation \eqref{eq1.1} then follows for  $\alpha=-1/2$.
Note that $
j_{\alpha}(\lambda t)$ is the eigenfunction of the following Sturm--Liouville problem:
\[
(t^{2\alpha+1}u_{\lambda}'(t))'+
\lambda^2 t^{2\alpha+1} u_{\lambda}(t)=0,\quad
u_{\lambda}(0)=1,\quad u_{\lambda}'(0)=0,\quad t,\lambda\in
\mathbb{R}_{+}.
\]

\subsection*{$m$-Logan problem for the Jacobi
transform}
In this paper, we solve the analog of Problem~C for the Jacobi
transform with the  kernel $\varphi_{\lambda}(t)$ being the eigenfunction of the Sturm--Liouville problem

\begin{equation}\label{eq1.3}
\begin{gathered}
(\Delta(t)\varphi_{\lambda}'(t))'+
(\lambda^2+\rho^2)\Delta(t)\varphi_{\lambda}(t)=0,
\\
\varphi_{\lambda}(0)=1,\quad \varphi'_{\lambda}(0)=0,
\end{gathered}
\end{equation}
with the
weight function  given by
\begin{equation*}
\Delta(t)=\Delta^{(\alpha,\beta)}(t)=2^{2\rho}(\sh t)^{2\alpha+1}(\ch
t)^{2\beta+1},\quad t\in \mathbb{R}_{+},
\end{equation*}
where
\[
\alpha\geq\beta\geq-1/2,\quad \rho=\alpha+\beta+1.
\]
The following representation for the Jacobi function is known
\[
\varphi_{\lambda}(t)=\varphi_{\lambda}^{(\alpha,\beta)}(t)=
F\Bigl(\frac{\rho+i\lambda}{2},\frac{\rho-i\lambda}{2};\alpha+1;-(\sh
t)^{2}\Bigr),
\]
where $F(a,b;c;z)$ is
the Gauss hypergeometric function.

For the precise definitions of direct and inverse Jacobi transforms, see the next section.
In the case $\alpha=\beta=-1/2$ we have $\Delta(t)=1$ and the Jacobi
transform is reduced to the cosine Fourier transform.

Set, for a real-valued continuous  function $f$ on $\mathbb{R}_{+}$,
\[
\Lambda(f)=\Lambda(f,\mathbb{R}_{+})=\sup\,\{\lambda>0\colon\, f(\lambda)>0\}
\]
{($\Lambda (f)=0$ if $f\leq 0$) and}
$$\Lambda_{m}(f)=\Lambda((-1)^{m-1}f).$$

Consider the class $\mathcal{L}_{m}(\tau,\mathbb{R}_{+})$, $m\in\mathbb{N}$, $\tau>0$, of real-valued even functions $f\in
C_b(\mathbb{R}_{+})$ such that

\smallbreak
(1) $f$ is the Jacobi transform of a measure
\begin{equation}\label{def-sigma1--}
f(\lambda)=\int_{0}^{2\tau}\varphi_{\lambda}(t)\,d\nu(t), \quad \lambda\in
\mathbb{R}_{+}, 
\end{equation}
where $\nu$ is a nontrivial function of bounded variation non-decreasing in some neighborhood of the origin;

\smallbreak
(2) if $m\ge 2$, then, additionally, $f\in L^1(\mathbb{R}_{+},\lambda^{2m-4}\,d\sigma)$
and there holds
\begin{equation}\label{orth}
\int_{0}^{\infty}\lambda^{2k}f(\lambda)\,d\sigma(\lambda)=0,\quad
k=0,1,\dots,m-2,
\end{equation}
where $\sigma$ is the spectral measure of the Sturm--Liouville problem \eqref{eq1.3}, that is,
\begin{equation}\label{def-sigma1}
d\sigma(\lambda)=d\sigma^{(\alpha,\beta)}(\lambda)=s(\lambda)\,d\lambda,
\end{equation}
where the spectral weight
\[
s(\lambda)=s^{(\alpha,\beta)}(\lambda)=(2\pi)^{-1}\Bigl|\frac{2^{\rho-i\lambda}\Gamma(\alpha+1)\Gamma(i\lambda)}
{\Gamma((\rho+i\lambda)/2)\Gamma((\rho+i\lambda)/2-\beta)}\Bigr|^{-2}.
\]

This class $\mathcal{L}_{m}(\tau,\mathbb{R}_{+})$ is not empty. In particular, we will show that it contains the function
\begin{equation}\label{extr}
f_m(\lambda)=\varphi_{\lambda}(\tau)F_m(\lambda),
\end{equation}
where
\begin{equation}\label{extr-}
F_m(\lambda)=\frac{\varphi_{\lambda}(\tau)}{(1-\lambda^2/\lambda_1^2(\tau))\cdots(1-\lambda^2/\lambda_m^2(\tau))}
\end{equation}
and $0<\lambda_{1}(t)<\dots<\lambda_{k}(t)<\cdots$ are the positive zeros of
$\varphi_{\lambda}(t)$ 
as a function in $\lambda$.

The $m$-Logan problem for Jacobi transform on the half-line  is formulated as follows.

\begin{probD}
Find 
\[
L_m(\tau,\mathbb{R}_{+})=\inf\{\Lambda_{m}(f)\colon\, f\in \mathcal{L}_{m}(\tau,\mathbb{R}_{+})\}.
\]
\end{probD}

\begin{rem}
In the case $\alpha=\beta=-1/2$ Problem D becomes Problem C. Even though the ideas to solve Problem D are similar to those we used in the solution of Problem C, the proof is far from being  just a generalization, 
since the weight function $\Delta(t)$ and the
spectral weight
$s(\lambda)$ have  completely
different behavior than in  the case of classical Fourier or Hankel transforms (see \cite{GIT19}).

\end{rem}

\subsection*{The main result
} 



\begin{thm}\label{thm-1}
%
Let $m\in\mathbb{N}$, $\tau>0$. Then
\[
L_m(\tau,\mathbb{R}_{+})=\lambda_m(\tau)
\]
and 
 the function $f_{m}$ is the unique extremizer up to multiplication by a
positive constant.
Moreover,  $f_{m}$ is positive definite with respect to the inverse Jacobi transform~and
\begin{equation}\label{orth-fam}
\int_{0}^{\infty}\lambda^{2k}f_{m}(\lambda)\,d\sigma(\lambda)=0,\quad
k=0,1,\dots,m-1.
\end{equation}
\end{thm}

\begin{rem}\label{rem-zam} We note
the inverse Jacobi transform $g_m(t)=\mathcal{J}^{-1}f_{m}(t)\ge 0$.
Furthermore, the function
$F_m$ given by (\ref{extr-}) is  positive
definite since
$G_m(t)=\mathcal{J}^{-1}F_{m}(t)$ 
 is non-negative and decreases on $[0,\tau]$,
and it has zero of multiplicity $2m-1$ at $t=\tau$.
The relationship between
 $g_m(t)$ and $G_m(t)$ is given by
 $g_m(t)=T^{\tau}G_m(t)$, where $T^{\tau}$ is the generalized translation
operator, see Section~\ref{sec-JH}.
\end{rem}



\subsection*{Structure of the paper}
The presentation follows our  paper \cite{GIT19}.
Section~\ref{sec-JH} contains some facts on the Jacobi harmonic analysis as well as a Gauss
quadrature formula with zeros of the Jacobi functions as nodes.
  In Section~\ref{sec-Cheb}, we prove that the Jacobi functions form the Chebyshev systems, which is used in the proof of Theorem~\ref{thm-1}.

In Section~\ref{sec-LJ}, we give the solution of the generalized Logan problem
for the Jacobi transform. 
 Using Theorem~\ref{thm-1}, in Section~\ref{sec-Hyp} we   solve
the multidimensional Logan problem for the
Fourier transform on the hyperboloid.

Finally,
Section~\ref{sec-JPD} is devoted to the problem on the minimal interval
containing  $n$ zeros of functions represented by the Jacobi transform
of a nonnegative bounded Stieltjes measure. Originally, such questions were investigated by Logan in \cite{Lo83c} for the cosine transform. 
It is worth mentioning that  extremizers in this problem and Problem~D are closely related. 



\section{Elements of  Jacobi harmonic analysis}\label{sec-JH}

Below we give some needed facts;  
see \cite{FlenKoor73,FlenKoor79,Koor75,Koor84,GorIva15}.


Let $\mathcal{E}^{\tau}$ be the class of even entire functions
$g(\lambda)$ of exponential type at most $\tau>0$, satisfying the estimate
$|g(\lambda)|\leq c_g\,e^{\tau|\mathrm{Im}\,\lambda|}$, $\lambda\in \mathbb{C}$.

The Jacobi function $\varphi_{\lambda}(t)$ is an even  analytic function of $t$ on $\mathbb{R}$ and it
belongs to the class $\mathcal{E}^{|t|}$ with respect to $\lambda$.  Moreover, the following
conditions hold:
\begin{equation}\label{eq2.1}
|\varphi_{\lambda}(t)|\leq 1,\quad \varphi_{0}(t)>0, \quad \lambda,t\in
\mathbb{R}.
\end{equation}

From the general properties of the eigenfunctions of the
Sturm--Liouville problem (see, for example, \cite{LevSar88}), one has that,
for $t>0$, $\lambda\in \mathbb{C}$,
\begin{equation}\label{eq2.2}
\varphi_{\lambda}(t)=\varphi_{0}(t)\prod_{k=1}^{\infty}\Bigl(1-\frac{\lambda^2}{\lambda_{k}^2(t)}\Bigr),
\end{equation}
where $0<\lambda_{1}(t)<\dots<\lambda_{k}(t)<\cdots$ are the positive zeros of
$\varphi_{\lambda}(t)$ as a function of $\lambda$.

We also have that $\lambda_{k}(t)=t_{k}^{-1}(t)$, where $t_{k}(\lambda)$ are the positive zeros
of the function $\varphi_{\lambda}(t)$
as a function of $t$.
 The zeros $t_{k}(\lambda)$, as well as the zeros $\lambda_{k}(t)$,
are continuous and 
 strictly  decreasing 
 \cite[Ch.~I,
\S\,3]{LevSar88}.


\subsection*{Properties of some special functions}
In what follows, we will need the asymptotic behavior of the Jacobi function and
spectral weight, see \cite{GorIva15}:
\begin{equation}\label{eq2.3}
\varphi_{\lambda}(t)=\frac{(2/\pi)^{1/2}}{(\Delta(t)s(\lambda))^{1/2}}
\Bigl(\cos \Bigl(\lambda t-\frac{\pi(\alpha+1/2)}{2}\Bigr)+e^{t|\mathrm{Im}\,\lambda|}O(|\lambda|^{-1})\Bigr),\quad |\lambda|\to +\infty,\quad t>0,
\end{equation}
\begin{equation}\label{eq2.4}
s(\lambda)=(2^{\rho+\alpha}\Gamma(\alpha+1))^{-2}\lambda^{2\alpha+1}(1+O(\lambda^{-1})),\quad
\lambda\to +\infty.
\end{equation}

From \eqref{eq2.3} and \eqref{eq2.4} it follows that, for fixed $t>0$ and uniformly on $\lambda\in \mathbb{R}_{+}$,
\begin{equation}\label{eq2.5}
|\varphi_{\lambda}(t)|\lesssim\frac{1}{(\lambda+1)^{\alpha+1/2}},
\end{equation}
where as usual $F_{1}\lesssim F_{2}$ means $F_{1}\le CF_{2}$.
Also we denote
$F_1\asymp F_2$ if ${C}^{-1}F_1\le F_2\le C F_1$ with $C\ge 1$.

In the Jacobi harmonic analysis, an important role is played by the function
\begin{equation}\label{eq2.6}
\psi_{\lambda}(t)=\psi_{\lambda}^{(\alpha,
\beta)}(t)=\frac{\varphi_{\lambda}^{(\alpha, \beta)}(t)}{\varphi_{0}^{(\alpha,
\beta)}(t)}=\frac{\varphi_{\lambda}(t)}{\varphi_{0}(t)},
\end{equation}
which  is the solution of the Sturm--Liouville problem
\begin{equation}\label{eq2.7}
(\Delta_{*}(t)\,\psi_{\lambda}'(t))'+
\lambda^2\Delta_{*}(t)\psi_{\lambda}(t)=0,\quad
\psi_{\lambda}(0)=1,\quad \psi_{\lambda}'(0)=0,
\end{equation}
where $\Delta_{*}(t)=\varphi_{0}^2(t)\Delta(t)$ is the modified weight function.

The positive zeros $0<\lambda_{1}^*(t)<\dots<\lambda_{k}^{*}(t)<\cdots$ of the
function $\psi_{\lambda}'(t)$ of $\lambda$ alternate with the zeros of the
function $\varphi_{\lambda}(t)$ \cite{GorIva15}:
\begin{equation}\label{eq2.8}
0<\lambda_{1}(t)<\lambda_{1}^{*}(t)<\lambda_{2}(t)<\dots<\lambda_{k}(t)<\lambda_{k}^{*}(t)<\lambda_{k+1}(t)<\cdots.
\end{equation}

For the derivative of the Jacobi function one has
\begin{equation}\label{eq2.9}
(\varphi_{\lambda}^{(\alpha,\beta)}(t))_t'=-\frac{(\rho^{2}+\lambda^{2})\sh t\ch
t}{2(\alpha+1)}\,\varphi_{\lambda}^{(\alpha+1,\beta+1)}(t).
\end{equation}
Moreover, according to \eqref{eq1.3},
\[
\bigl\{\Delta(t)\bigl(\varphi_{\mu}(t)\varphi'_{\lambda}(t)-\varphi_{\mu}'(t)\varphi_{\lambda}(t)\bigr)\bigr\}_t'=
(\mu^2-\lambda^2)\Delta(t)\varphi_{\mu}(t)\varphi_{\lambda}(t)
\]
and therefore,
\begin{equation}\label{eq2.10}
\int_{0}^{\tau}\Delta(t)\varphi_{\mu}(t)\varphi_{\lambda}(t)\,dt=
\frac{\Delta(\tau)\bigl(\varphi_{\mu}(t)\varphi'_{\lambda}(t)-\varphi_{\mu}'(\tau)\varphi_{\lambda}(\tau)\bigr)}
{\mu^2-\lambda^2}.
\end{equation}

\begin{lem}
For the Jacobi functions, the following recurrence formula
\begin{multline}\label{eq2.11}
\frac{(\lambda^2+(\alpha+\beta+3)^2)(\sh t\ch
t)^{2}}{4(\alpha+1)(\alpha+2)}\,\varphi_{\lambda}^{(\alpha+2,\beta+2)}(t)\\
=\frac{(\alpha+1)\ch^2t+(\beta+1)\sh^2t}{\alpha+1}\,\varphi_{\lambda}^{(\alpha+1,\beta+1)}(t)-
\varphi_{\lambda}^{(\alpha,\beta)}(t)
\end{multline}
and the formula for derivatives
\begin{equation}\label{eq2.12}
\bigl((\sh t)^{2\alpha+3}(\ch
t)^{2\beta+3}\varphi_{\lambda}^{(\alpha+1,\beta+1)}(t)\bigr)_t'=2(\alpha+1)(\sh t)^{2\alpha+1}(\ch
t)^{2\beta+1}\varphi_{\lambda}^{(\alpha,\beta)}(t)
\end{equation}
are valid.
\end{lem}
\begin{proof} Indeed, \eqref{eq2.11} and \eqref{eq2.12} are easily derived from
\eqref{eq1.3} and \eqref{eq2.9}. To prove \eqref{eq2.11}, we rewrite  \eqref{eq1.3} as
\[
\Delta(t)\varphi_{\lambda}''(t)+\Delta'(t)\varphi_{\lambda}'(t)+(\lambda^2+\rho^2)\Delta(t)\varphi_{\lambda}(t)=0,
\]
and then we replace the first and second derivatives  by the Jacobi functions using \eqref{eq2.9}.
To show \eqref{eq2.12}, we use \eqref{eq2.9} and \eqref{eq2.11}.
\end{proof}

 Many properties (e.g., inequality \eqref{eq2.1}) of the Jacobi function follow from the Mehler representation
\begin{equation}\label{eq2.13}
\varphi_{\lambda}(t)=\frac{c_{\alpha}}{\Delta(t)}
\int_{0}^{t}A_{\alpha,\beta}(s,t)\cos{}(\lambda s)\,ds,\quad
A_{\alpha,\beta}(s,t)\ge 0,
\end{equation}
where $c_{\alpha}=\frac{\Gamma(\alpha+1)}{\Gamma(1/2)\,\Gamma(\alpha+1/2)}$ and
\begin{multline*}
A_{\alpha,\beta}(s,t)=2^{\alpha+2\beta+5/2}\sh{}(2t)\ch^{\beta-\alpha}t
\bigl(\ch{}(2t)-\ch{}(2s)\bigr)^{\alpha-1/2}
\\
{}\times F\Bigl(\alpha+\beta,\alpha-\beta;\alpha+\frac{1}{2};\frac{\ch t-\ch s}{2\ch
t}\Bigr).
\end{multline*}
We will need some properties of
the following functions
\begin{equation}\label{eq2.14}
\begin{gathered}
\eta_{\varepsilon}(\lambda)=\psi_{\lambda}(\varepsilon)=\frac{\varphi_{\lambda}(\varepsilon)}{\varphi_0(\varepsilon)},\quad
\varepsilon>0,\quad \lambda\ge 0,\\
\eta_{m-1,\varepsilon}(\lambda)=(-1)^{m-1}\Bigl(\eta_{\varepsilon}(\lambda)-\sum_{k=0}^{m-2}
\frac{\eta_{\varepsilon}^{(2k)}(0)}{(2k)!}\,\lambda^{2k}\Bigr),\quad m\ge 2,\\
\rho_{m-1,\varepsilon}(\lambda)=\frac{(2m-2)!\,
\eta_{m-1,\varepsilon}(\lambda)}{(-1)^{m-1}\eta_{\varepsilon}^{(2m-2)}(0)}.
\end{gathered}
\end{equation}


\begin{lem}\label{lem-2}
For any $\varepsilon>0$, $m\ge 2$, $\lambda\in \mathbb{R}_{+}$,
\[
\eta_{m-1,\varepsilon}(\lambda)\ge 0,\quad (-1)^{m-1}\eta_{\varepsilon}^{(2m-2)}(0)>0,
\]
\[
\rho_{m-1,\varepsilon}(\lambda)\ge 0,
\quad \lim\limits_{\varepsilon\to 0}\rho_{m-1,\varepsilon}(\lambda)=\lambda^{2m-2}.
\]
\end{lem}

\begin{proof} Using the inequality
\[
(-1)^{m-1}\Bigl(\cos \lambda- \sum_{k=0}^{m-2}\frac{(-1)^{k}
\lambda^{2k}}{(2k)!}\Bigr)\ge 0,
\]
and \eqref{eq2.13},  we get
\[
\eta_{m-1,\varepsilon}(\lambda)\ge 0,\quad (-1)^{m-1}\eta_{\varepsilon}^{(2m-2)}(0)=\frac{c_{\alpha}}{\Delta(\varepsilon)\varphi_{0}(\varepsilon)}
\int_{0}^{\varepsilon}A_{\alpha,\beta}(s,\varepsilon)s^{2m-2}\,ds>0.
\]
Hence, $\rho_{m-1,\varepsilon}(\lambda)\ge 0$.
For any $\lambda\in \mathbb{R}_{+}$,
\[
\eta_{\varepsilon}(\lambda)
=\sum_{k=0}^{\infty}\frac{\eta_{\varepsilon}^{(2k)}(0)}{(2k)!}\,\lambda^{2k}.
\]
By differentiating  equality \eqref{eq2.2} in $\lambda$ and  substituting $\lambda=0$, we obtain
\[
(-1)^k\eta_{\varepsilon}^{(2k)}(0)
=2^k\sum_{i_1=1}^{\infty}\frac{1}{\lambda_{i_1}^2(\varepsilon)}
\sum_{i_2\neq i_1}^{\infty} \frac{1}{\lambda_{i_2}^2(\varepsilon)}\dots
\sum_{i_k\neq i_1,\dots,i_{k-1}}^{\infty}\frac{1}{\lambda_{i_k}^2(\varepsilon)}.
\]
Hence,
\[
|\eta_{\varepsilon}''(0)|=2\sum_{i=1}^{\infty}\frac{1}{\lambda_{i}^2(\varepsilon)},
\]
and for $k\ge m$
\[
\Bigl|\frac{\eta_{\varepsilon}^{(2k)}(0)}{\eta_{\varepsilon}^{(2m-2)}(0)}\Bigr|\leq |\eta_{\varepsilon}''(0)|^{k-m+1}.
\]
Therefore,
\begin{align*}
\frac{|\rho_{m-1,\varepsilon}(\lambda)-\lambda^{2m-2}|}{(2m-2)!}&=\Bigl|\frac{
\eta_{m-1,\varepsilon}(\lambda)}{\eta_{\varepsilon}^{(2m-2)}(0)}-\frac{\lambda^{2m-2}}{(2m-2)!}\Bigr|=
\Bigl|\frac{\eta_{m,\varepsilon}(\lambda)}{\eta_{\varepsilon}^{(2m-2)}(0)}\Bigr|\\&\le
\sum_{k=m}^{\infty}\Bigl|\frac{\eta_{\varepsilon}^{(2k)}(0)}{\eta_{\varepsilon}^{(2m-2)}(0)}
\Bigr|\frac{\lambda^{2k}}{(2k)!}\le |\eta_{\varepsilon}''(0)|\sum_{k=m}^\infty|
\eta_{\varepsilon}''(0)|^{k-m}\frac{\lambda^{2k}}{(2k)!}.
\end{align*}
It remains to show that
\begin{equation*}
\lim\limits_{\varepsilon\to 0}|\eta_{\varepsilon}''(0))|=0.
\end{equation*}
Zeros $\lambda_{k}(\varepsilon)$ monotonically  decrease  on $\varepsilon$ and, for any $k$,
$\lim\limits_{\varepsilon\to 0}\lambda_{k}(\varepsilon)=\infty$. In view of  \eqref{eq2.3},
we have 
$\lambda_{k}(1)\asymp k$
as $k\to\infty$. Finally, the result follows from 
\[
|\eta_{\varepsilon}''(0))|\le \sum_{k=1}^N\frac{1}{\lambda_{k}^2(\varepsilon)}+
\sum_{k=N+1}^{\infty}\frac{1}{\lambda_{k}^2(1)}\lesssim \sum_{k=1}^N\frac{1}{\lambda_{k}^2(\varepsilon)}+\frac{1}{N}.
\]
\end{proof}

\subsection*{Jacobi transforms, translation, and positive definiteness}
As usual, if $X$ is a manifold  with the positive
measure $\rho$, then by $L^{p}(X,d\rho)$, $p\ge 1$, we denote the
Lebesgue space with the finite norm
$\|f\|_{p,d\rho}=\bigl(\int_{X}|f|^p\,d\rho\bigr)^{1/p}$. For $p=\infty$,
$C_b(X)$ is the space of continuous bounded functions with norm
$\|f\|_{\infty}=\sup_{X}|f|$. Let $\mathrm{supp}\,f$ be the support of a
function~$f$.


{Let $t,\lambda\in \mathbb{R}_{+}$, $d\mu(t)=\Delta(t)\,dt$ and $d\sigma(\lambda)$ be the spectral measure
\eqref{def-sigma1}.
Then}
 $L^{2}(\mathbb{R}_{+}, d\mu)$ and $L^{2}(\mathbb{R}_{+}, d\sigma)$ are
 Hilbert spaces with the inner products
\[
(g, G)_{\mu}=\int_{0}^{\infty}g(t)\overline{G(t)}\,d\mu(t),\quad (f,
F)_{\sigma}=\int_{0}^{\infty}f(\lambda)\overline{F(\lambda)}\,d\sigma(\lambda).
\]

The main concepts of harmonic analysis in $L^{2}(\mathbb{R}_{+}, d\mu)$ and $L^{2}(\mathbb{R}_{+}, d\sigma)$
  are the direct and inverse Jacobi transforms, namely,
\[
\mathcal{J}g(\lambda)=\mathcal{J}^{(\alpha,\beta)}g(\lambda)=\int_{0}^{\infty}g(t)\varphi_{\lambda}(t)\,d\mu(t)
\]
and
\[
\mathcal{J}^{-1}f(t)=(\mathcal{J}^{(\alpha,\beta)})^{-1}f(t)=\int_{0}^{\infty}f(\lambda)\varphi_{\lambda}(t)\,d\sigma(\lambda).
\]
We recall a few basic facts.
If $g\in L^{2}(\mathbb{R}_{+}, d\mu)$, $f\in L^{2}(\mathbb{R}_{+}, d\sigma)$, then
$\mathcal{J}g\in L^{2}(\mathbb{R}_{+}, d\sigma)$, $\mathcal{J}^{-1}f\in
L^{2}(\mathbb{R}_{+}, d\mu)$ and $g(t)=\mathcal{J}^{-1}(\mathcal{J}g)(t)$,
$f(\lambda)=\mathcal{J}(\mathcal{J}^{-1}f)(\lambda)$ in the mean square sense
and, moreover, the Parseval relations hold.

In addition, if $g\in L^{1}(\mathbb{R}_{+}, d\mu)$, then $\mathcal{J}g\in C_b(\mathbb{R}_{+})$
and $\|\mathcal{J}g\|_{\infty}\leq \|g\|_{1,d\mu}$. If $f\in
L^{1}(\mathbb{R}_{+}, d\sigma)$, then $\mathcal{J}^{-1}f\in C_b(\mathbb{R}_{+})$
and $\|\mathcal{J}^{-1}f\|_{\infty}\leq \|f\|_{1,d\sigma}$.

Furthermore, assuming $g\in L^{1}(\mathbb{R}_{+}, d\mu)\cap C_b(\mathbb{R}_{+})$, $\mathcal{J}g\in
L^{1}(\mathbb{R}_{+}, d\sigma)$, one has,  for any $t\in\mathbb{R}_{+}$,
\[
g(t)=\int_{0}^{\infty}\mathcal{J}g(\lambda)\varphi_{\lambda}(t)\,d\sigma(\lambda).
\]
Similarly, assuming $f\in L^{1}(\mathbb{R}_{+}, d\sigma)\cap C_b(\mathbb{R}_{+})$,
$\mathcal{J}^{-1}f\in L^{1}(\mathbb{R}_{+}, d\mu)$, one has, for any
$\lambda\in\mathbb{R}_{+}$,
\[
f(\lambda)=\int_{0}^{\infty}\mathcal{J}^{-1}f(t)\varphi_{\lambda}(t)\,d\mu(t).
\]

Let $\mathcal{B}_1^{\tau},$ $\tau>0$, be the Bernstein class of even entire functions from $\mathcal{E}^{\tau}$, whose
restrictions to $\mathbb{R}_{+}$ belong to $L^{1}(\mathbb{R}_{+}, d\sigma)$. For
functions from the class $\mathcal{B}_1^{\tau}$, the following Paley--Wiener theorem is
valid.


\begin{lem}[\cite{Koor75,GI19}]\label{lem-3}
A function $f$ belongs to $\mathcal{B}_1^{\tau}$ if and only if
\[
f\in L^{1}(\mathbb{R}_{+}, d\sigma)\cap C_b(\mathbb{R}_{+})\quad \text{and}\quad
\mathrm{supp}\,\mathcal{J}^{-1}f\subset[0,\tau].
\]
Moreover, there holds
\[
f(\lambda)=\int_{0}^{\tau}\mathcal{J}^{-1}f(t)\varphi_{\lambda}(t)\,d\mu(t),\quad \lambda\in \mathbb{R}_{+}.
\]
\end{lem}

Let us now  discuss  the generalized translation operator and convolution.
In view of \eqref{eq2.1}, the generalized translation operator in
$L^{2}(\mathbb{R}_{+}, d\mu)$ is defined by \cite[Sect.~4]{FlenKoor73} 
\[
T^tg(x)=\int_{0}^{\infty}\varphi_{\lambda}(t)
\varphi_{\lambda}(x)\mathcal{J}g(\lambda)\,d\sigma(\lambda),\quad
t,x\in \mathbb{R}_{+}.
\]
If $\alpha\ge\beta\ge-1/2$, $\alpha>-1/2$, the following integral representation holds:
\begin{equation}\label{eq2.17}
T^tg(x)=\int_{|t-x|}^{t+x}g(u)K(t,x,u)\,d\mu(u),
\end{equation}
where the kernel $K$ is nonnegative and symmetric. 
 Note that for $\alpha=\beta=-1/2$, we arrive at
$T^tg(x)=(g(t+x)+g(|t-x|))/2$.

Using representation \eqref{eq2.17}, we can extend the generalized translation
operator to the spaces $L^{p}(\mathbb{R}_{+}, d\mu)$, $1\leq p\leq\infty$, and,
for any $t\in\mathbb{R}_{+}$, we have $\|T^t\|_{p\to p}=1$
\cite[Lemma~5.2]{FlenKoor73}.

The operator $T^t$ possesses the following properties:

\smallbreak
(1) \ $\text{if}\ g(x)\geq 0,\ \text{then}\ T^tg(x)\geq 0$;

\smallbreak
(2) \ $T^t\varphi_{\lambda}(x)=\varphi_{\lambda}(t)\varphi_{\lambda}(x), \
\mathcal{J}(T^tg)(\lambda)=\varphi_{\lambda}(t)\mathcal{J}g(\lambda)$;

\smallbreak
(3) \ $T^tg(x)=T^xg(t), \ T^t1=1$;

\smallbreak
(4) \ $\text{if}\ g \in L^{1}(\mathbb{R}^d_{+}, d\mu),\ \text{then}\
\int_{0}^{\infty}T^tg(x)\,d\mu(x)=\int_{0}^{\infty}g(x)\,d\mu(x)$;

\smallbreak
(5) \ $\text{if}\ \mathrm{supp}\,g\subset[0, \delta],\ \text{then} \
\mathrm{supp}\,T^tg\subset[0, \delta+t]$.

\smallbreak
Using the generalized translation operator $T^t$, we can define the
convolution 
and
positive-definite functions.
Following \cite{FlenKoor73}, we set
\[
(g\ast G)_{\mu}(x)=\int_{0}^{\infty}T^tg(x)G(t)\,d\mu(t).
\]


\begin{lem}[{\cite[Sect.~5]{FlenKoor73}}]\label{lem-5}
 If $g, G\in L^{1}(\mathbb{R}_{+}, d\mu)$, then $\mathcal{J}(g\ast
G)_{\mu}=\mathcal{J}g\,\mathcal{J}G$.
Moreover, if $\mathrm{supp}\,g\subset[0, \delta]$, $\mathrm{supp}\,G\subset[0, \tau]$,
then $\mathrm{supp}\,(g\ast G)_{\mu}\subset[0, \delta+\tau]$.
\end{lem}

An even  continuous function $g$ is called positive definite with respect to
Jacobi transform~$\mathcal{J}$ if for any $N$
\[
\sum_{i,j=1}^Nc_i\overline{c_j}\,T^{x_i}g(x_j)\ge 0,\quad
\forall\,c_1,\dots,c_N\in\mathbb{C},\quad
\forall\,x_1,\dots,x_N\in\mathbb{R}_{+},
\]
or, equivalently, the matrix  $(T^{x_i}g(x_j))_{i,j=1}^{N}$ is positive semidefinite.
If a continuous function~$g$ has the representation
\[
g(x)=\int_{0}^{\infty}\varphi_{\lambda}(x)\,d\nu(\lambda),
\]
where $\nu$ is a non-decreasing function of bounded variation, then $g$ is positive definite.
Indeed, using the property (2) for the operator $T^{t}$, we obtain
\begin{align*}
\sum_{i,j=1}^Nc_i\overline{c_j}\,T^{x_i}g(x_j)&=\int_{0}^{\infty}\sum_{i,j=1}^Nc_i\overline{c_j}
\,T^{x_i}\varphi_{\lambda}(x_j)\,d\nu(\lambda)
\\&=\int_{0}^{\infty}\sum_{i,j=1}^Nc_i\overline{c_j}
\,\varphi_{\lambda}(x_i)\varphi_{\lambda}(x_j)\,d\nu(\lambda)=\int_{0}^{\infty}
\Bigl|\sum_{i=1}^Nc_i \,\varphi_{\lambda}(x_i)\Bigr|^2\,d\nu(\lambda)\ge 0.
\end{align*}
If $g\in L^{1}(\mathbb{R}_{+}, d\mu)$, then a sufficient condition for positive definiteness of $g$ is
$\mathcal{J}g(\lambda)\ge 0$.

We can also define the generalized translation operator in
$L^{2}(\mathbb{R}_{+}, d\sigma)$ by
\[
S^{\eta}f(\lambda)=\int_{0}^{\infty}\varphi_{\eta}(t)
\varphi_{\lambda}(t)\mathcal{J}^{-1}f(t)\,d\mu(t),\quad
\eta,\lambda\in \mathbb{R}_{+}.
\]

Then, for $\alpha\ge\beta\ge-1/2$, $\alpha>-1/2$, the following integral representation holds:
\begin{equation}\label{shift-oper}
S^{\eta}f(\lambda)=\int_{0}^{\infty}f(\zeta)L(\eta,\lambda,\zeta)\,d\sigma(\zeta),
\end{equation}
where the kernel
\[
L(\eta,\lambda,\zeta)=\int_{0}^{\infty}\varphi_{\eta}(t)
\varphi_{\lambda}(t)\varphi_{\zeta}(t)\,d\mu(t),\quad
\int_{0}^{\infty}L(\eta,\lambda,\zeta)\,d\sigma(\zeta)=1,
\]
is nonnegative continuous and symmetric, 
\cite{FlenKoor79}. Using \eqref{shift-oper}, we can extend the generalized
translation operator to the spaces $L^{p}(\mathbb{R}_{+}, d\sigma)$, $1\leq
p\leq\infty$, and, for any $\eta\in\mathbb{R}_{+}$, $\|S^\eta\|_{p\to p}=1$
\cite{FlenKoor79}.

One has

\smallbreak
(1) \ $\text{if}\ f(\lambda)\geq 0,\ \text{then}\ S^{\eta}f(\lambda)\geq 0$;

\smallbreak
(2) \ $S^{\eta}\varphi_{\lambda}(t)=\varphi_{\eta}(t)\varphi_{\lambda}(t), \
\mathcal{J}^{-1}( S^{\eta}f)(t)=\varphi_{\eta}(t)\mathcal{J}^{-1}f(t)$;

\smallbreak
(3) \ $S^{\eta}f(\lambda)=S^{\lambda}f(\eta), \ S^{\eta}1=1$;

\smallbreak
(4) \ $\text{if}\ f \in L^{1}(\mathbb{R}^d_{+}, d\sigma),\ \text{then}\
\int_{0}^{\infty}S^{\eta}f(\lambda)\,d\sigma(\lambda)=\int_{0}^{\infty}f(\lambda)\,d\sigma(\lambda)$.

\smallbreak
The function $\zeta\mapsto L(\eta,\lambda,\zeta)$ is analytic for
$|\mathrm{Im}\,\zeta|<\rho$. Hence, the restriction of this function to $\mathbb{R}_{+}$ has no compact
support, in contrast with the function $x\mapsto K(t,s,x)$ in~\eqref{eq2.17}. 

Similarly to above, we define
\[
(f\ast F)_{\sigma}(\lambda)=\int_{0}^{\infty}S^{\eta}f(\lambda)F(\eta)\,d\sigma(\eta).
\]
If $f, F\in L^{1}(\mathbb{R}_{+}, d\sigma)$, then $\mathcal{J}^{-1}(f\ast
F)_{\sigma}=\mathcal{J}^{-1}f\,\mathcal{J}^{-1}F$. 

An even continues function  is called positive definite with respect to
the inverse Jacobi transform $\mathcal{J}^{-1}$ if
\[
\sum_{i,j=1}^Nc_i\overline{c_j}\,S^{\lambda_i}f(\lambda_j)\ge 0,\quad
\forall\,c_1,\dots,c_N\in\mathbb{C},\quad
\forall\,\lambda_1,\dots,\lambda_N\in\mathbb{R}_{+},
\]
or, equivalently, the matrix  $(S^{\lambda_i}f(\lambda_j))_{i,j=1}^{N}$ is positive semidefinite.
If a continuous function~$f$ has the representation
\[
f(\lambda)=\int_{0}^{\infty}\varphi_{\lambda}(t)\,d\nu(t),
\]
where $\nu$ is a non-decreasing function of bounded variation, then $f$ is positive definite.
If $f\in L^{1}(\mathbb{R}_{+}, d\sigma)$, then a sufficient condition for positive definiteness is
$\mathcal{J}^{-1}f(t)\ge 0$.

\subsection*{Gauss quadrature and lemmas on entire functions
}
In what follows, we will need the Gauss quadrature formula on the half-line
for entire functions of exponential type.

\begin{lem}[\cite{GorIva15}]
For an arbitrary function $f\in \mathcal{B}_1^{2\tau}$, the Gauss quadrature formula with
positive weights holds:
\begin{equation}\label{eq2.16}
\int_{0}^{\infty}f(\lambda)\,d\sigma(\lambda)=
\sum_{k=0}^{\infty}\gamma_{k}(\tau)f(\lambda_{k}(\tau)).
\end{equation}
The series in \eqref{eq2.16} converges absolutely.
\end{lem}

\begin{lem}[\cite{GIT19}]\label{lem-6}
Let $\alpha>-1/2$. There exists an even entire function $\omega_{\alpha}(z)$
of exponential type $2$, positive for $z>0$, and such that
\begin{align*}
\omega_{\alpha}(x)&\asymp x^{2\alpha+1},\quad x\to +\infty,\\
|\omega_{\alpha}(iy)|&\asymp y^{2\alpha+1}e^{2y},\quad y\to +\infty.
\end{align*}
\end{lem}
The next lemma  is an easy consequence of Akhiezer's result \cite[Appendix
VII.10]{Le80}.
\begin{lem}\label{lem-7}
Let $F$ be an even entire function of exponential type $\tau>0$ bounded
on~$\mathbb{R}$. Let $\Omega$ be an even entire function of finite exponential
type, let all the zeroes of $\Omega$ be zeros of $F$, and let, for some $m\in
\mathbb{Z}_{+}$,
\[
\liminf_{y\to +\infty}e^{-\tau y}y^{2m}|\Omega(iy)|>0.
\]
Then the function $F(z)/\Omega(z)$ is an even polynomial of degree at most $2m$.
\end{lem}

\section{Chebyshev systems of Jacobi functions}\label{sec-Cheb}

Let $I$ be an interval on $\mathbb{R}_{+}$. By $N_{I}(g)$ we denote  the number of zeros of a continuous function $g$ on interval $I$, counting multiplicity.
A family of real-valued functions $\{\varphi_{k}(t)\}_{k=1}^{\infty}$
defined on an interval $I$ is a Chebyshev system
(T-system) if for any $n\in \mathbb{N}$ and any nontrivial linear combination
$$p(t)=\sum_{k=1}^{n}A_{k}\varphi_{k}(t),$$ there holds $N_{I}(p)\le n-1$, see, e.g., \cite[Chap.~II]{Ac04}.

Our goal is to prove that some systems, constructed with the help of  Jacobi functions, are the Chebyshev systems.
We will use the convenient for us version of Sturm's theorem on zeros of linear combinations of
eigenfunctions of the Sturm--Liouville problem, see~\cite{BH17}.

\begin{thm}[\cite{BH17}]\label{thm-3}

Let $\{u_{k}\}_{k=1}^{\infty}$ be the system of eigenfunctions associated to
eigenvalues $\xi_1<\xi_2<\dots$ of the following Sturm--Liouville problem
on the interval $[0,\tau]$:
\begin{equation}\label{eq5.1}
(wu')'+\xi wu=0,\quad u'(0)=0,\quad \cos \theta\,u(\tau)+\sin \theta\,u'(\tau)=0,
\end{equation}
where $\xi=\lambda^{2}+\lambda_{0}^{2}$,
$\xi_{k}=\lambda_{k}^{2}+\lambda_{0}^{2}$, $w\in C[0,\tau]$, $w\in
C^{1}(0,\tau)$, $w>0$ on $(0,\tau)$, $\theta\in [0,\pi/2]$.

Then for any non-trivial real polynomial of the form
\[
p=\sum_{k=m}^{n}a_{k}u_{k},\quad m,n\in \mathbb{N},\quad m\le n,
\]
we have
\[
m-1\le N_{(0,\tau)}(p)\le n-1.
\]
In particular, every $k$-th eigenfunction $u_{k}$ has exactly $k-1$ simple zeros.
\end{thm}


As above we assume that $\tau>0$, $\alpha\ge\beta\ge -1/2$, $\alpha>-1/2$, $\varphi_{\lambda}(t)=\varphi_{\lambda}^{(\alpha,\beta)}(t)$,
$\psi_{\lambda}(t)=\psi_{\lambda}^{(\alpha,\beta)}(t)$, $\lambda_k(t)=\lambda_k^{(\alpha,\beta)}(t)$, and
$\lambda_k^{*}(t)=\lambda_k^{*\,(\alpha,\beta)}(t)$
for $k\in \mathbb{N}$. Let $0<\mu_1(t)<\mu_2(t)<\dots$ be the positive zeros of the function $\varphi_{\lambda}'(t)$ of $\lambda$.

\begin{thm}\label{thm-4}
\textup{(i)} The families of the Jacobi functions
\begin{equation}\label{eq5.2}
\{\varphi_{\lambda_k(\tau)}(t)\}_{k=1}^{\infty},\quad \{\varphi_{\mu_k(\tau)}(t)\}_{k=1}^{\infty}
\end{equation}
form Chebyshev systems on $[0,\tau)$ and  $(0,\tau)$, respectively.

\textup{(ii)}~The families of the Jacobi functions
\[
\{\varphi_{\mu_k(\tau)}'(t)\}_{k=1}^{\infty},\quad \{\varphi_{\lambda_k(\tau)}'(t)\}_{k=1}^{\infty},
\quad\{\varphi_{\mu_k(\tau)}(t)-\varphi_{\mu_k(\tau)}(\tau)\}_{k=1}^{\infty}
\]
form Chebyshev systems on $(0,\tau)$.
\end{thm}
\begin{proof}
 The families \eqref{eq5.2} are  the systems of eigenvalues for the Sturm--Liouville
problem \eqref{eq5.1} when 
 $\lambda_{0}=\rho=
\alpha+\beta+1
 $, $w(t)=\Delta(t)$, and $\theta=0,\pi/2$.
Then, by Theorem \ref{thm-3}, the statement of part (i) is valid for the interval $(0,\tau)$.
In order to include the endpoint $t=0$ for the family  $\{\varphi_{\lambda_k(\tau)}(t)\}_{k=1}^{\infty}$, we first take care of  part (ii).

Since
\[
\varphi_{\lambda}'(t)=-\frac{(\lambda^2+\rho^2)\sh t\ch t}{2(\alpha+1)}\varphi_{\lambda}^{(\alpha+1,\beta+1)}(t),\quad \rho>0,
\]
it is sufficiently to prove that the families
$\{\varphi_{\mu_k(\tau)}^{(\alpha+1,\beta+1)}(t)\}_{k=1}^{\infty}$ and
$\{\varphi_{\lambda_k(\tau)}^{(\alpha+1,\beta+1)}(t)\}_{k=1}^{\infty}$ are the Chebyshev systems on $(0,\tau)$.

For the family $\{\varphi_{\mu_k(\tau)}^{(\alpha+1,\beta+1)}(t)\}_{k=1}^{\infty}$, this again follows from Theorem \ref{thm-3} since
it
 is the system of eigenvalues of the
Sturm--Liouville problem \eqref{eq5.1} with
 $\lambda_{0}=\rho,
 $
 $w(t)=\Delta^{(\alpha+1,\beta+1)}(t)$, and $\theta=0$.

For the second family $\{\varphi_{\lambda_k(\tau)}^{(\alpha+1,\beta+1)}(t)\}_{k=1}^{\infty}$, let us
assume that the polynomial
\[
p(t)=\sum_{k=1}^{n}a_{k}\varphi_{\lambda_k(\tau)}^{(\alpha+1,\beta+1)}(t)
\]
has $n$ zeros on $(0,\tau)$.
We consider the function $g(t)=(\sh t)^{2\alpha+2}(\ch t)^{2\beta+2}p(t)$. It has $n+1$ zeros including $t=0$.
By Rolle's theorem, for a smooth real function $g$ one has $N_{(0,\tau)}(g')\ge
N_{(0,\tau)}(g)-1\ge n$ (see \cite{BH17}).
In light of \eqref{eq2.13}, we obtain
\[
g'(t)=2(\alpha+1)(\sh t)^{2\alpha+2}(\ch t)^{2\beta+2}\sum_{k=1}^{n}a_{k}\varphi_{\lambda_k(\tau)}(t).
\]
This contradicts the fact that
$\{\varphi_{\lambda_k(\tau)}(t)\}_{k=1}^{\infty}$ is the Chebyshev system on $(0,\tau)$.

To show that $\{\varphi_{\mu_k(\tau)}(t)-\varphi_{\mu_k(\tau)}(\tau)\}_{k=1}^{\infty}$ is the Chebyshev system on $(0,\tau)$, assume that
$p(t)=\sum_{k=1}^{n}a_{k}(\varphi_{\mu_k(\tau)}(t)-\varphi_{\mu_k(\tau)}(\tau))$ has $n$
zeros on $(0,\tau)$. Taking into account the zero $t=\tau$, its derivative $p'(t)=\sum_{k=1}^{n}a_{k}\varphi_{\mu_k(\tau)}'(t)$
has at least $n$ zeros on $(0,\tau)$. This
cannot be true because
 $\{\varphi_{\mu_k(\tau)}'(t)\}_{k=1}^{\infty}$ is the Chebyshev system on $(0,\tau)$.

Now we are in a position to show that the first system in \eqref{eq5.2} is Chebyshev on $[0,\tau)$.
If $p(t)=\sum_{k=1}^{n}a_{k}\varphi_{\lambda_k(\tau)}(t)$ has $n$ zeros on $[0,\tau)$, then always $p(0)=0$.
Moreover, $p(\tau)=0$. Therefore, $p'(t)$ has at least $n$ zeros on $(0,\tau)$, which is impossible
since $p'(t)=\sum_{k=1}^{n}a_{k}\varphi_{\lambda_k(\tau)}'(t)$ and
$\{\varphi_{\lambda_k(\tau)}'(t)\}_{k=1}^{\infty}$ is the Chebyshev system on $(0,\tau)$.

\end{proof}

\begin{thm}\label{thm-5}
\textup{(i)} The families of the Jacobi functions
\begin{equation}\label{eq5.3}
\{\psi_{\lambda_k(\tau)}(t)\}_{k=1}^{\infty},\quad
\{1\}\cup\{\psi_{\lambda_k^{*}(\tau)}(t)\}_{k=1}^{\infty}
\end{equation}
form Chebyshev systems on $(0,\tau)$ and $[0,\tau]$, respectively.

\textup{(ii)}~The families of the Jacobi functions
\[
\{\psi_{\lambda_k^{*}(\tau)}'(t)\}_{k=1}^{\infty},\quad
\{\psi_{\lambda_k^{*}(\tau)}(t)-\psi_{\lambda_k^{*}(\tau)}(\tau)\}_{k=1}^{\infty}
\]
form Chebyshev systems on $(0,\tau)$.
\end{thm}

\begin{proof}
The families \eqref{eq5.3} are the systems of eigenvalues for the Sturm--Liouville problem \eqref{eq5.1}
in the case  $\lambda_{0}=0$, $\Delta_{*}(t)=\varphi_0^2(t)\Delta(t)$, and $\theta=0,\pi/2$.
Then the statement of part (i) is valid for the interval $(0,\tau)$.
In order to include the endpoints, we first prove part (ii).

Let $w(t)=\Delta_{*}(t)=\varphi_0^2(t)\Delta(t)$, $W(t)=\int_{0}^tw(s)\,ds$, $w_0(t)=W^2(t)w^{-1}(t)$. It is known \cite{GorIva15}
that $v_{\lambda}(t)=-w(t)W^{-1}(t)\lambda^{-2}\psi_{\lambda}'(t)$ is the eigenfunction of the Sturm--Liouville problem
\[
(w_0v')'+\lambda^2w_0v=0,\quad v'(0)=0.
\]
Hence, the family $\{v_{\lambda_k^{*}(\tau)}'(t)\}_{k=1}^{\infty}$ is the system of eigenvalues for the Sturm--Liouville problem
\[
(w_0v')'+\lambda^2w_0v=0,\quad v'(0)=0,\quad v(\tau)=0.
\]
By Theorem~\ref{thm-3}, the family $\{v_{\lambda_k^{*}(\tau)}'(t)\}_{k=1}^{\infty}$ and the family
$\{\psi_{\lambda_k^{*}(\tau)}'(t)\}_{k=1}^{\infty}$ are the Chebyshev systems on $(0,\tau)$.

To prove that $\{\psi_{\lambda_k^{*}(\tau)}(t)-\psi_{\lambda_k^{*}(\tau)}(\tau)\}_{k=1}^{\infty}$ forms the Chebyshev system on $(0,\tau)$, we assume that
$p(t)=\sum_{k=1}^{n}a_{k}(\psi_{\lambda_k^{*}(\tau)}(t)-\psi_{\lambda_k^{*}(\tau)}(\tau))$ has $n$
zeros on $(0,\tau)$. Taking into account the zero $t=\tau$, its derivative
$p'(t)=\sum_{k=1}^{n}a_{k}\psi_{\lambda_k^{*}(\tau)}'(t)$
has at least $n$ zeros on $(0,\tau)$. This contradicts the fact that
 $\{\psi_{\lambda_k^{*}(\tau)}'(t)\}_{k=1}^{\infty}$ is the Chebyshev system on $(0,\tau)$.

Now we are in a position to show that the second system in \eqref{eq5.3} is Chebyshev on $[0,\tau]$.
If $p(t)=\sum_{k=0}^{n-1}a_{k}\psi_{\lambda_k^{*}(\tau)}(t)$
(we assume $\lambda_0^{*}(\tau)=0$) has $n$ zeros on $[0,\tau]$, then one of the endpoints is zero. Then
$p'(t)=\sum_{k=1}^{n-1}a_{k}\psi_{\lambda_k^{*}(\tau)}'(t)$ has at least $n-1$ zeros on $(0,\tau)$, which is impossible for Chebyshev system
$\{\psi_{\lambda_k^{*}(\tau)}'(t)\}_{k=1}^{\infty}$.
\end{proof}

\section{Proof of Theorem~\ref{thm-1}
}\label{sec-LJ}

Below we give a 
 solution of the generalized $m$-Logan problem for the
Jacobi transform. As above, let $m\in\mathbb{N}$ and $\tau>0$. For brevity, we
denote
\[
\lambda_{k}=\lambda_k(\tau),\quad \gamma_k=\gamma_k(\tau).
\]

We need the following

\begin{lem}
Let $f(\lambda)$ be a nontrivial function from
$\mathcal{L}_{m}(\tau,\mathbb{R}_{+})$ such that $\Lambda_{m}(f)<\infty$. Then
\begin{equation}\label{H-1}
f\in L^1(\mathbb{R}_{+},\lambda^{2m-2}\,d\sigma),\quad
(-1)^{m-1}\int_{0}^{\infty}\lambda^{2m-2}f(\lambda)\,d\sigma(\lambda)\ge
0,
\end{equation}
\end{lem}

\begin{proof}
Let $m=1$. Let $\varepsilon>0$, $\chi_{\varepsilon}(t)$ be the characteristic function of
the 
  interval~$[0, \varepsilon]$,
\[
\Psi_{\varepsilon}(t)=c_{\varepsilon}^{-2}(\chi_{\varepsilon}\ast
\chi_{\varepsilon})_{\mu}(t),\quad c_{\varepsilon}=\int_{0}^{\varepsilon}\,d\mu.
\]
By Lemma~\ref{lem-5}, $\mathrm{supp}\,\Psi_{\varepsilon}\subset
[0,2\varepsilon]$. According to the properties (1)--(4) of the generalized
translation operator $T^t$ and Lemma~\ref{lem-5}, we have
\[
\Psi_{\varepsilon}(t)\geq 0,\quad
\mathcal{J}\Psi_{\varepsilon}(\lambda)=c_{\varepsilon}^{-2}(\mathcal{J}\chi_{\varepsilon}(\lambda))^2,
\]
\[
\int_{0}^{\infty}\Psi_{\varepsilon}(t)\,d\mu(t)=
c_{\varepsilon}^{-2}\int_{0}^{\infty}\chi_{\varepsilon}(x)\int_{0}^{\infty}T^x\chi_{\varepsilon}(t)\,d\mu(t)\,d\mu(x)=1.
\]
Since $\chi_{\varepsilon}\in L^{1}(\mathbb{R}_{+}, d\mu)\cap L^{2}(\mathbb{R}_{+},
d\mu)$, $\mathcal{J}\chi_{\varepsilon}\in L^{2}(\mathbb{R}_{+}, d\sigma)\cap
C_b(\mathbb{R}_{+})$, and
\[
|\mathcal{J}\chi_{\varepsilon}(\lambda)|\leq c_{\varepsilon},\quad\lim\limits_{\varepsilon\to
0}c_{\varepsilon}^{-1}\mathcal{J}\chi_{\varepsilon}(\lambda)=\lim\limits_{\varepsilon\to
0}c_{\varepsilon}^{-1}\int_{0}^{\varepsilon}\varphi_{\lambda}(t)\,d\mu(t)=1,
\]
we obtain
\[
\mathcal{J}\Psi_{\varepsilon}\in L^{1}(\mathbb{R}_{+}, d\sigma)\cap
C_b(\mathbb{R}_{+}), \quad 0\leq \mathcal{J}\Psi_{\varepsilon}(\lambda)\leq
1,\quad \lim\limits_{\varepsilon\to
0}\mathcal{J}\Psi_{\varepsilon}(\lambda)=1.
\]

The fact that $\mathcal{L}_{1}(\tau,\mathbb{R}_{+})\subset
L^1(\mathbb{R}_{+},d\nu_{\alpha})$ can be verified with the help of  Logan's method from
\cite[Lemma]{Lo83b}.
Indeed, let
$f\in \mathcal{L}_{1}(\tau,\mathbb{R}_{+})$ be given by
\eqref{def-sigma1--}.
Taking into account that $d\nu\ge 0$ in some neighborhood of the origin,
we derive,  for sufficiently small $\varepsilon>0$, that
\begin{align*}
0&\le \int_{0}^{2\varepsilon}\Psi_{\varepsilon}(t)\,d\nu(t)=
\int_{0}^{\infty}\Psi_{\varepsilon}(t)\,d\nu(t)=
\int_{0}^{\infty}f(\lambda)\mathcal{J}\Psi_{\varepsilon}(\lambda)\,d\sigma(\lambda)\\
&=\int_{0}^{\lambda_{1}(f)}f(\lambda)\mathcal{J}\Psi_{\varepsilon}(\lambda)(\lambda)\,d\sigma(\lambda)-
\int_{\lambda_{1}(f)}^{\infty}|f(\lambda)|\mathcal{J}\Psi_{\varepsilon}(\lambda)\,d\sigma(\lambda).
\end{align*}
This gives
\[
\int_{\lambda_{1}(f)}^{\infty}|f(\lambda)|\mathcal{J}\Psi_{\varepsilon}(\lambda)\,d\sigma(\lambda)\le
\int_{0}^{\lambda_{1}(f)}f(\lambda)\mathcal{J}\Psi_{\varepsilon}(\lambda)\,d\sigma(\lambda)\le
\int_{0}^{\lambda_{1}(f)}|f(\lambda)|\,d\sigma(\lambda).
\]
Letting $\varepsilon\to 0$, by Fatou's lemma, we have
\[
\int_{\lambda_{1}(f)}^{\infty}|f(\lambda)|\,d\sigma(\lambda)\le
\int_{0}^{\lambda_{1}(f)}|f(\lambda)|\,d\sigma(\lambda)<\infty.
\]

Let $m\ge 2$. In light of the definition of the class $\mathcal{L}_{m}(\tau,\mathbb{R}_{+})$, we have
$f\in L^1(\mathbb{R}_{+},d\sigma)$ and $d\nu(t)\ge 0$ on segment $[0,\varepsilon]$,
therefore $d\nu(t)=\mathcal{J}^{-1}f(t)d\mu(t)$ and $\mathcal{J}^{-1}f(\varepsilon)\ge 0$ for sufficiently small $\varepsilon$.

Consider the function $\rho_{m-1,\varepsilon}(\lambda)$ defined by \eqref{eq2.14}.
Using Lemma \ref{lem-2}, the orthogonality property \eqref{orth}, and
the equality $(-1)^{m}f(\lambda)=|f(\lambda)|$ for $\lambda\ge \Lambda_{m}(f)$, we arrive at
\begin{align}
(-1)^{m-1}\int_{0}^{\infty}\rho_{m-1,\varepsilon}(\lambda)f(\lambda)\,d\sigma(\lambda)&=
\frac{(2m-2)!}{(-1)^{m-1}\varphi_0(\varepsilon)\psi_{\varepsilon}^{(2m-2)}(0)}
\int_{0}^{\infty}f(\lambda)\varphi_{\lambda}(\varepsilon)\,d\sigma(\lambda)
\notag\\
&=\frac{(2m-2)!}{(-1)^{m-1}\varphi_0(\varepsilon)\psi_{\varepsilon}^{(2m-2)}(0)}\,\mathcal{J}^{-1}f(\varepsilon)\ge
0. \label{eq21}
\end{align}
Thus,
\[
(-1)^{m}\int_{\Lambda_{m}(f)}^{\infty}\rho_{m-1,\varepsilon}(\lambda)f(\lambda)\,d\sigma(\lambda)\le
(-1)^{m-1}\int_{0}^{\Lambda_{m}(f)}\rho_{m-1,\varepsilon}(\lambda)f(\lambda)\,d\sigma(\lambda).
\]
Taking into account \eqref{eq21}, Lemma \ref{lem-2}, and Fatou's lemma, we have
\begin{align*}
&(-1)^{m}\int_{\Lambda_{m}(f)}^{\infty}\lambda^{2m-2}f(\lambda)\,d\sigma(\lambda)
=(-1)^{m}\int_{\Lambda_{m}(f)}^{\infty}\lim_{\varepsilon\to
0}\rho_{m-1,\varepsilon}(\lambda)f(\lambda)\,d\sigma(\lambda)\\
&\qquad \le\liminf_{\varepsilon\to
0}{}(-1)^{m}\int_{\Lambda_{m}(f)}^{\infty}\rho_{m-1,\varepsilon}(\lambda)f(\lambda)\,d\sigma(\lambda)\\
&\qquad \le\liminf_{\varepsilon\to
0}{}(-1)^{m-1}\int_0^{\Lambda_{m}(f)}\rho_{m-1,\varepsilon}(\lambda)f(\lambda)\,d\sigma(\lambda)\\
&\qquad =(-1)^{m-1}\int_0^{\Lambda_{m}(f)}\lim_{\varepsilon\to
0}\rho_{m-1,\varepsilon}(\lambda)f(\lambda)\,d\sigma(\lambda)\\
&\qquad =(-1)^{m-1}\int_0^{\Lambda_{m}(f)}\lambda^{2m-2}f(\lambda)\,d\sigma(\lambda)<\infty.
\end{align*}
Therefore, we obtain that $f\in L^1(\mathbb{R}_{+},\lambda^{2m-2}\,d\sigma)$ and, using
\eqref{eq21}, the condition
$
(-1)^{m-1}\int_{0}^{\infty}\lambda^{2m-2}f(\lambda)\,d\sigma(\lambda)\ge
0$ holds and \eqref{H-1} follows.

\end{proof}

\begin{proof}[Proof of Theorem~\ref{thm-1}]
The proof is divided into several steps.

\subsection*{Lower bound}
Firstly, we establish the inequality
\begin{equation*}
L_m(\tau,\mathbb{R}_{+})\ge \lambda_m.
\end{equation*}

Consider a function $f\in \mathcal{L}_{m}(\tau,\mathbb{R}_{+})$. Let us show that $\lambda_m\le \Lambda_{m}(f)$.
Assume the converse, i.e., $\Lambda_{m}(f)<\lambda_{m}$. Then $(-1)^{m-1}f(\lambda)\le 0$ for
$\lambda\ge \Lambda_{m}(f)$. Using \eqref{H-1} implies $\lambda^{2m-2}f(\lambda)\in
\mathcal{B}_{1}^{2\tau}$. Therefore, by Gauss' quadrature formula
\eqref{eq2.16} and \eqref{orth}, we obtain
\begin{align}
0&\le(-1)^{m-1}\int_{0}^{\infty}\lambda^{2m-2}f(\lambda)\,d\sigma(\lambda) =
(-1)^{m-1}\int_{0}^{\infty}\prod_{k=1}^{m-1}(\lambda^2-\lambda_{k}^2)
f(\lambda)\,d\sigma(\lambda) \notag\\
&=
(-1)^{m-1}\sum_{s=m}^{\infty}
\gamma_{s}f(\lambda_{s})\prod_{k=1}^{m-1}(\lambda_{s}^2-\lambda_{k}^2)\le 0.
\label{eq22}
\end{align}
Therefore, $\lambda_{s}$ for
 $s\ge m$ are zeros of multiplicity $2$ for $f$.
Similarly, applying
Gauss' quadrature formula for $f$, we derive that
\begin{equation}\label{eq23}
0=\int_{0}^{\infty}
\prod_{\substack{k=1\\ k\ne s}}^{m-1}(\lambda^2-\lambda_{k}^2)f(\lambda)\,d\sigma(\lambda) =
\gamma_{s}\prod_{\substack{k=1\\ k\ne s}}^{m-1}(\lambda_{s}^2-\lambda_{k}^2)f(\lambda_{s}),\quad
s=1,\dots,m-1.
\end{equation}
Therefore, $\lambda_{s}$ for $s=1,\dots,m-1$ are zeros of $f$.

From $f\in L^1(\mathbb{R}_{+},\,d\sigma)$ and asymptotic behavior of  $s(\lambda)$ given by \eqref{eq2.4}
it follows that $f\in L^1(\mathbb{R}_{+},\lambda^{2\alpha+1}\,d\lambda)$.
Consider the function $\omega_{\alpha}(\lambda)$ from Lemma \ref{lem-6} and set
\[
W(\lambda)=\omega_{\alpha}(\lambda)f(\lambda),\quad
\Omega(\lambda)=\frac{\omega_{\alpha}(\lambda)\varphi_{\lambda}^2(\tau)}
{\prod_{k=1}^{m-1}(1-\lambda^2/\lambda_{k}^2)}.
\]
Then  functions $W$ and $\Omega$ are even and have
exponential type $4$.
Since $\omega_{\alpha}(\lambda)\asymp \lambda^{2\alpha+1}$, $\lambda\to +\infty$,
then $W\in L^{1}(\mathbb{R})$ and $W$ is bounded on $\mathbb{R}$.

From \eqref{eq2.3} and Lemma \ref{lem-6}, we have
\[
|\Omega(iy)|\asymp y^{-2m+2}e^{4y},\quad y\to +\infty.
\]
Taking into account that
all zeros of $\Omega(\lambda)$ are also zeros of $F(\lambda)$ and  Lemma \ref{lem-7}, we arrive at
\[
f(\lambda)=\frac{\varphi_{\lambda}^2(\tau)\sum_{k=0}^{m-1}c_k\lambda^{2k}}
{\prod_{k=1}^{m-1}(1-\lambda^2/\lambda_{k}^2)},
\]
where $c_{k}\neq 0$ for some $k$. By \eqref{eq2.3},
$\varphi_{\lambda}^2(\tau)=O(\lambda^{-2\alpha-1})$ as
$\lambda\to +\infty$, and by \eqref{eq2.4} $\varphi_{\lambda}^2(\tau)\notin L^{1}(\mathbb{R}_{+},d\sigma)$.
This contradicts $f\in L^1(\mathbb{R}_{+},\lambda^{2m-2}\,d\sigma)$. Thus, $\Lambda_{m}(f)\ge \lambda_{m}$
and $L_m(\tau, \mathbb{R}_{+})\ge \lambda_{m}$.

\subsection*{Extremality of $f_{m}$}
Now we consider the function $f_{m}$ given by \eqref{extr}. Note that by \eqref{eq2.5} we have
the estimate $f_{m}(\lambda)=O(\lambda^{-2\alpha-1-2m})$ as $\lambda\to +\infty$ and hence
$f_{m}\in L^1(\mathbb{R}_{+},\lambda^{2m-2}\,d\sigma)$. Moreover, $f_m$ is an entire function of exponential
type $2\tau$ and $\Lambda_{m}(f_{m})=\lambda_{m}$.

To verify facts that $f_{m}(\lambda)$ is positive definite with respect to the inverse Jacobi transform
and the property \eqref{orth-fam} holds,
we first note that Gauss' quadrature formula implies \eqref{orth-fam}.
From the property (2) of the generalized translation operator $T^t$, one has that $g_m(t)=T^{\tau}G_m(t)$; see Remark \ref{rem-zam}.
Since $T^t$ is  a positive operator, to show the inequality $g_m(t)\ge 0$, it is enough to prove $G_m(t)\ge 0$.
This will be shown in the next subsection.

Thus, we have shown that  $f_{m}$ is the extremizer.
The uniqueness of $f_m$ will be proved later.

\subsection*{Positive definiteness of $F_{m}$}
Our goal here is to find the function $G_m(t)$ such that $F_m(\lambda)=\mathcal{J}G_m(\lambda)$ and show that it is nonnegative.

For fixed $\mu_1,\dots,\mu_k\in \mathbb{R}$, consider the polynomial
\[
\omega_k(\mu)=\omega(\mu,\mu_1,\dots,\mu_k)=\prod_{i=1}^k(\mu_i-\mu), \quad \mu\in \mathbb{R}.
\]
Then
\[
\frac{1}{\omega_k(\mu)}=\sum_{i=1}^k\frac{1}{\omega_k'(\mu_i)(\mu_i-\mu)}.
\]
Setting $k=m$, $\mu=\lambda^2$, $\mu_i=\lambda_i^2$, $i=1,\dots,m$, we have
\begin{equation}\label{eq24}
\frac{1}{\prod_{i=1}^{m}(1-\lambda^2/\lambda_{i}^2)}
=
\prod_{i=1}^{m}\lambda_{i}^2 \frac{1}{\omega_{m}(\lambda^2)}
=
\prod_{i=1}^{m}\lambda_{i}^2\sum_{i=1}^{m}\frac{1}{\omega_{m}'(\lambda_{i}^2)(\lambda_{i}^2-\lambda^2)}
=
\sum_{i=1}^{m}\frac{A_i}{\lambda_{i}^2-\lambda^2},
\end{equation}
where
\begin{equation}\label{eq25}
\omega_{m}'(\lambda_i^2)=\prod_{\substack{j=1\\ j\ne i}}^{m}(\lambda_{j}^2-\lambda_{i}^2)\quad\mbox{and}\quad A_i=
\frac{\prod_{{j=1}}^{m}\lambda_{j}^2}
{\omega_m'(\lambda_{i}^2)}.
\end{equation}
Note that
\begin{equation}\label{eq26}
\mathrm{sign}\,A_i=(-1)^{i-1}.
\end{equation}

For simplicity,  we set
\[
\Phi_{i}(t):=\varphi_{\lambda_i}(t),\quad i=1,\dots,m,
\]
and observe that $\Phi_{i}(t)$ are eigenfunctions
and $\lambda_{i}^2+\rho^2$ are eigenvalues of the following
 Sturm--Liouville problem on $[0,1]$:
\begin{equation}\label{eq27}
(\Delta(t) u'(t))'+(\lambda^2+\rho^2)\Delta(t)u(t)=0,\quad u'(0)=0,\quad u(\tau)=0.
\end{equation}

Let $\chi(t)$ be the characteristic function of $[0,\tau]$.
In light of
  \eqref{eq2.10} and $\Phi_{i}(\tau)=0$, we have
\[
\int_{0}^{\infty}\Phi_{i}(t)\varphi_{\lambda}(t)\chi(t)\Delta(t)\,dt=
\int_0^{\tau}\Phi_{i}(t)\varphi_{\lambda}(t)\Delta(t)\,dt=
-\frac{\Delta(\tau)\Phi_{i}'(\tau)\varphi_{\lambda}(\tau)}
{\lambda_i^2-\lambda^2},
\]
or, equivalently,
\begin{equation}\label{eq28}
\mathcal{J}\Bigl(-\frac{\Phi_{i}\chi}{\Delta(\tau)\Phi_{i}'(\tau)}\Bigr)(\lambda)=
\frac{\varphi_{\lambda}(\tau)}{\lambda_{i}^2-\lambda^2}.
\end{equation}
It is important to note that
\begin{equation}\label{eq29}
\mathrm{sign}\,\Phi_{i}'(\tau)=(-1)^i.
\end{equation}

Now we examine the following polynomial in eigenfunctions $\Phi_{i}(t)$:
\begin{equation}\label{eq30}
p_{m}(t)=-\frac{1}{\Delta(\tau)}\sum_{i=1}^{m}\frac{A_i}{\Phi_{i}'(\tau)}\,\Phi_{i}(t)=:
\sum_{i=1}^{m}B_i\Phi_{i}(t).
\end{equation}
By virtue of \eqref{eq26} and \eqref{eq29}, we derive that $B_i>0$,
$p_{m}(0)>0$, and $p_{m}(\tau)=0$. Furthermore, because of \eqref{eq24} and \eqref{eq28},
\begin{equation}\label{p-g}
\mathcal{J}(p_{m}\chi)(\lambda)=
\frac{\varphi_{\lambda}(\tau)}{\prod_{i=1}^{m}(1-\lambda^2/\lambda_{i}^2)}=: F_{m}(\lambda).
\end{equation}
Hence,
it suffices to verify that $p_{m}(t)\ge 0$ on $[0,\tau]$.
Define the Vandermonde determinant
$
\Delta(\mu_1,\dots,\mu_k)=\prod_{1\le j<i\le k}^{k}(\mu_i-\mu_j),
$ 
then
\[
\frac{\Delta(\mu_1,\dots,\mu_k)}{\omega_k'(\mu_i)}=
(-1)^{i-1}\Delta(\mu_1,\dots,\mu_{i-1},\mu_{i+1},\dots,\mu_k).
\]
From \eqref{eq24} and \eqref{eq25}, we have
\begin{align}
p_{m}(t)&=-c_{1}\sum_{i=1}^{m}(-1)^{i-1}\Delta(\lambda_1^2,\dots,\lambda_{i-1}^2,\lambda_{i+1}^2,\dots,\lambda_{m}^2)\,
\frac{\Phi_{i}(t)}{\Phi_{i}'(\tau)} \notag\\
&=-c_{1}\begin{vmatrix}
\frac{\Phi_{1}(t)}{\Phi_{1}'(\tau)}&\dots&\frac{\Phi_{m}(t)}{\Phi_{m}'(\tau)}\\
1&\dots&1\\
\lambda_1^2&\dots &\lambda_{m}^2\\
\hdotsfor[2]{3}\\
\lambda_1^{2m-4}&\dots &\lambda_{m}^{2m-4}
\end{vmatrix},\quad c_{1}=\frac{\prod_{{j=1}}^{m}\lambda_{j}^2}{\Delta(\lambda^2_1,\dots,\lambda^2_m)}>0.
\label{eq31}
\end{align}
Here and in what follows, if $m=1$ we consider  only
the $(1,1)$ entries of the matrices.

We now show  that
\begin{equation}\label{eq32}
\begin{vmatrix}
\frac{\Phi_{1}(t)}{\Phi_{1}'(\tau)}&\dots&\frac{\Phi_{m}(t)}{\Phi_{m}'(\tau)}\\
1&\dots&1\\
\lambda_1^2&\dots &\lambda_{m}^2\\
\hdotsfor[2]{3}\\
\lambda_1^{2m-4}&\dots &\lambda_{m}^{2m-4}
\end{vmatrix}=(-1)^{\frac{(m-1)(m-2)}{2}}
\begin{vmatrix}
\frac{\Phi_{1}(t)}{\Phi_{1}'(\tau)}&\dots&\frac{\Phi_{m}(t)}{\Phi_{m}'(\tau)}\\
1&\dots&1\\
\frac{\Phi_{1}^{(3)}(\tau)}{\Phi_{1}'(\tau)}&\dots&\frac{\Phi_{m}^{(3)}(\tau)}{\Phi_{m}'(\tau)}\\
\hdotsfor[2]{3}\\
\frac{\Phi_{1}^{(2m-3)}(\tau)}{\Phi_{1}'(\tau)}&\dots&\frac{\Phi_{m}^{(2m-3)}(\tau)}{\Phi_{m}'(\tau)}
\end{vmatrix}.
\end{equation}
Using \eqref{eq27}, we get
\[
\Phi_{i}''(t)+\Delta'(t)\Delta^{-1}(t)\Phi_{i}'(t)+(\lambda^2_i+\rho^2)\Phi_{i}(t)=0.
\]
By Leibniz's rule,
\[
\Phi_{i}^{(s+2)}(t)+\Delta'(t)\Delta^{-1}(t)\Phi_{i}^{(s+1)}(t)+(s(\Delta'(t)\Delta^{-1}(t))'+\lambda_i^2+\rho^2)
\Phi_{i}^{(s)}(t)
\]
\[
+\sum_{j=1}^{s-1}\binom{s}{j-1}(\Delta'(t)\Delta^{-1}(t))^{(s+1-j)}\Phi_{i}^{(j)}(t),
\]
which implies for $t=\tau$ that
\[
\Phi_{i}^{(s+2)}(\tau)=-\Delta'(\tau)\Delta^{-1}(\tau)\Phi_{i}^{(s+1)}(\tau)-(s(\Delta'(\tau)\Delta^{-1}(\tau))'+\lambda_i^2+\rho^2)
\Phi_{i}^{(s)}(\tau)
\]
\[
-\sum_{j=1}^{s-1}\binom{s}{j-1}(\Delta'(\tau)\Delta^{-1}(\tau))^{(s+1-j)}\Phi_{i}^{(j)}(\tau),\quad \Phi_{i}^{(0)}(\tau)=\Phi_{i}(\tau)=0.
\]
Assuming $s=0,1$ we obtain
\[
\Phi_{i}''(\tau)=-\Delta'(\tau)\Delta^{-1}(\tau)\Phi_{i}'(\tau),
\]
\[ \Phi_{i}'''(\tau)=((\Delta'(\tau)\Delta^{-1}(\tau))^2-(\Delta'(\tau)\Delta^{-1}(\tau))'+\rho^2+\lambda_i^2)\Phi_{i}'(\tau).
\]
By induction, we then derive for $k=0,1,\dots$
\[
\Phi_{i}^{(2k+1)}(\tau)=\Phi_{i}'(\tau)\sum_{j=0}^{k}a_{kj}\lambda_{i}^{2j},\quad
\Phi_{i}^{(2k+2)}(\tau)=\Phi_{i}'(\tau)\sum_{j=0}^{k}b_{kj}\lambda_{i}^{2j},
\]
where $a_{kj}$, $b_{kj}$ depend on $\alpha,\beta,\tau$ and do not depend of $\lambda_i$, and, moreover, $a_{kk}=(-1)^k$.
This yields for
$k=1,2,\dots$ that
\begin{equation}\label{eq34}
\frac{\Phi_{i}^{(2k)}(\tau)}{\Phi_{i}'(\tau)}=\sum_{s=1}^{k}c_{0s}\,
\frac{\Phi_{i}^{(2s-1)}(\tau)}{\Phi_{i}'(\tau)}
\end{equation}
and
\begin{equation}\label{eq35}
\frac{\Phi_{i}^{(2k+1)}(\tau)}{\Phi_{i}'(\tau)}=\sum_{s=1}^{k}c_{1s}\,
\frac{\Phi_{i}^{(2s-1)}(\tau)}{\Phi_{i}'(\tau)}+(-1)^k\lambda_i^{2k},
\end{equation}
where $c_{0s},c_{1s}$ do not depend of $\lambda_i$.
The latter implies \eqref{eq32} since
\[
\begin{vmatrix}
\frac{\Phi_{1}(t)}{\Phi_{1}'(\tau)}&\dots&\frac{\Phi_{m}(t)}{\Phi_{m}'(\tau)}\\
1&\dots&1\\
\frac{\Phi_{1}^{(3)}(\tau)}{\Phi_{1}'(\tau)}&\dots&\frac{\Phi_{m}^{(3)}(\tau)}{\Phi_{m}'(\tau)}\\
\hdotsfor[2]{3}\\
\frac{\Phi_{1}^{(2m-3)}(\tau)}{\Phi_{1}'(\tau)}&\dots&\frac{\Phi_{m}^{(2m-3)}(\tau)}{\Phi_{m}'(\tau)}
\end{vmatrix}=
\begin{vmatrix}
\frac{\Phi_{1}(t)}{\Phi_{1}'(\tau)}&\dots&\frac{\Phi_{m}(t)}{\Phi_{m}'(\tau)}\\
1&\dots&1\\
(-1)\lambda_1^2&\dots &(-1)\lambda_{m}^2\\
\hdotsfor[2]{3}\\
(-1)^{m-2}\lambda_1^{2m-4}&\dots &(-)^{m-2}\lambda_{m}^{2m-4}
\end{vmatrix}.
\]
Further, taking into account \eqref{eq31} and \eqref{eq32}, we derive
\begin{equation}\label{zv}
p_{m}^{(1)}(\tau)=p_{m}^{(3)}(\tau)=\cdots=p_{m}^{(2m-1)}(\tau)=0.
\end{equation}
Therefore, by \eqref{eq30} and \eqref{eq34}, we obtain for $j=1,\dots,m-1$ that
\begin{align*}
p_{m}^{(2j)}(\tau)&=-\frac{1}{\Delta(\tau)}\sum_{i=1}^{m}A_i\,\frac{\Phi_{i}^{(2j)}(\tau)}
{\Phi_{1}'(\tau)}=-\frac{1}{\Delta(\tau)}\sum_{i=1}^{m}A_i
\sum_{s=1}^{j}c_{0s}\,\frac{\Phi_{i}^{(2s-1)}(\tau)}
{\Phi_{1}'(\tau)}\\
&=-\frac{1}{\Delta(\tau)}\sum_{s=1}^{j}c_{0s}
\sum_{i=1}^{m}A_i\,\frac{\Phi_{i}^{(2s-1)}(\tau)}{\Phi_{1}'(\tau)}=
\sum_{s=1}^{j}c_{0s}p^{(2s-1)}(\tau)=0.
\end{align*}
Together with (\ref{zv}) this implies that the zero $t=\tau$ of the polynomial
$p_{m}(t)$ has multiplicity $2m-1$. Then taking into account \eqref{p-g}, the same also holds for $G_m(t)$.

The next step is to prove that $p_{m}(t)$ does not have zeros on $[0,\tau)$ and hence $p_{m}(t)>0$ on $[0,\tau)$, which implies that $G_{m}(t)\ge 0$ for $t\ge0$.
We will use the facts that $\{\Phi_{i}(t)\}_{i=1}^{m}$ is the Chebyshev system on the interval $(0,\tau)$ (see Theorem~\ref{thm-4} below) and
any polynomial of degree $m$ on $(0,\tau)$ has at most $m-1$ zeros, counting multiplicity.

We  consider the polynomial
\begin{equation}\label{eq37}
p(t,\varepsilon)=
\begin{vmatrix}
\frac{\Phi_{1}(t)}{\Phi_{1}'(\tau)}&\dots&\frac{\Phi_{m}(t)}{\Phi_{m}'(\tau)}\\
\frac{\Phi_{1}(\tau-\varepsilon)}{(-\varepsilon)\Phi_{1}'(\tau)}&\dots&
\frac{\Phi_{m}(\tau-\varepsilon)}{(-\varepsilon)\Phi_{m}'(\tau)}\\
\frac{\Phi_{1}(\tau-2\varepsilon)}{(-2\varepsilon)^3\Phi_{1}'(\tau)}&\dots&
\frac{\Phi_{m}(\tau-2\varepsilon)}{(-2\varepsilon)^3\Phi_{m}'(\tau)}\\
\hdotsfor[2]{3}\\
\frac{\Phi_{1}(\tau-(m-1)\varepsilon)}{(-(m-1)\varepsilon)^{2m-3}\Phi_{1}'(\tau)}&\dots&
\frac{\Phi_{m}(\tau-(m-1)\varepsilon)}{(-(m-1)\varepsilon)^{2m-3}\Phi_{m}'(\tau)}
\end{vmatrix}.
\end{equation}
For any $0<\varepsilon<\tau/(m-1)$, it has $m-1$ zeros at the points
$t_j=\tau-j\varepsilon$, $j=1,\dots,m-1$.
Letting $\varepsilon\to
0$, we observe that
the polynomial $\lim\limits_{\varepsilon\to 0}p(t,\varepsilon)$ does not have zeros on $(0,\tau)$.
If we demonstrate that
\begin{equation}\label{zv1}
\lim\limits_{\varepsilon\to 0}p(t,\varepsilon)=c_2p_{m}(t),
\end{equation}
with some $c_2>0$, then there holds that 
the polynomial $p_{m}(t)$ is strictly positive on $[0,\tau)$.

To prove \eqref{zv1}, we use  Taylor's formula:  for $j=1,\dots,m-1$,
\[
\frac{\Phi_{i}(\tau-j\varepsilon)}{(-j\varepsilon)^{2j-1}\Phi_{i}'(\tau)}=
\sum_{s=1}^{2j-2}\frac{\Phi_{i}^{(s)}(\tau)}{s!\,(-j\varepsilon)^{2j-1-s}\Phi_{i}'(\tau)}+
\frac{\Phi_{i}^{(2j-1)}(\tau)+o(1)}{(2j-1)!\,\Phi_{i}'(\tau)}.
\]
Using formulas \eqref{eq34} and \eqref{eq35} and progressively subtracting the row $j$ from the row $j-1$
in the determinant \eqref{eq37}, we have
\[
p(t,\varepsilon)=\frac{1}{\prod_{j=1}^{m-1}(2j-1)!}
\begin{vmatrix}
\frac{\Phi_{1}(t)}{\Phi_{1}'(\tau)}&\dots&\frac{\Phi_{m}(t)}{\Phi_{m}'(\tau)}\\
1+o(1)&\dots&
1+o(1)\\
\frac{\Phi_{1}^{(3)}(\tau)+o(1)}{\Phi_{1}'(\tau)}&\dots&
\frac{\Phi_{m}^{(3)}(\tau)+o(1)}{\Phi_{m}'(\tau)}\\
\hdotsfor[2]{3}\\
\frac{\Phi_{1}^{(2m-3)}(\tau)+o(1)}{\Phi_{1}'(\tau)}&\dots&
\frac{\Phi_{m}^{(2m-3)}(\tau)+o(1)}{\Phi_{m}'(\tau)}
\end{vmatrix}.
\]
Finally, in light of  \eqref{eq31} and \eqref{eq32}, we arrive at \eqref{zv1}.

\subsection*{Monotonicity of $G_m$}
The polynomial $p(t,\varepsilon)$ vanishes at $m$ points:
$t_{j}=\tau-j\varepsilon$, $j=1,\dots,m-1$, and $t_{m}=\tau$, thus its derivative
$p'(t,\varepsilon)$ has $m-1$ zeros on the interval $(\tau-\varepsilon, \tau)$.

In virtue of \eqref{eq2.7},
\[
\Phi_{i}'(t)=-\frac{(\rho^{2}+\lambda_i^{2})\sh t\ch
t}{2(\alpha+1)}\,\varphi_{\lambda_i}^{(\alpha+1,\beta+1)}(t),\quad
t\in [0,\tau].
\]

This and Theorem~\ref{thm-4} imply that 
$\{\Phi_{i}'(t)\}_{i=1}^{m}$ is the Chebyshev system on
$(0,\tau)$. Therefore, $p'(t,\varepsilon)$ does not have zeros on
$(0,\tau-\varepsilon]$. Then for $\varepsilon\to 0$ we derive that
$p_{m}'(t)$ does not have zeros on $(0,\tau)$. Since $p_{m}(0)>0$ and $p_{m}(\tau)=0$, then $p_{m}'(t)<0$ on $(0,\tau)$.
Thus,
$p_{m}(t)$ and $G_m(t)$ are decreasing on the interval
$[0,\tau]$.

\subsection*{Uniqueness of the extremizer $f_{m}$}
We will use Lemmas \ref{lem-6} and \ref{lem-7}. Let $f(\lambda)$ be an extremizer and $\Lambda_{m}(f)=\lambda_m$. Consider the functions
\[
F(\lambda)=\omega_{\alpha}(\lambda)f(\lambda),\quad
\Omega(\lambda)=\omega_{\alpha}(\lambda)f_{m}(\lambda),
\]
where $f_{m}$ is defined in \eqref{extr} and $\omega_{\alpha}$ is given in Lemma \ref{lem-6}.

Note that all zeros of $\Omega(\lambda)$ are also zeros of $F(\lambda)$.
Indeed, we have $(-1)^{m-1}f(\lambda)\le 0$ for $\lambda\ge \lambda_{m}$ and
$f(\lambda_{m})=0$ (otherwise $\Lambda_{m}(f)<\lambda_{m}$, which is a contradiction). This and \eqref{eq22} imply that
the points $\lambda_{s}$, $s\ge m+1$, are double zeros of $f$. By \eqref{eq23}, we also have that $f(\lambda_{s})=0$
for $s=1,\dots,m-1$ and therefore the function $f$ has zeros (at least, of order one) at the points $\lambda_{s}$, $s=1,\dots,m$.

Using asymptotic relations given in Lemma \ref{lem-6}, we derive that $F(\lambda)$ is the entire
function of exponential type, integrable on real line and therefore bounded.
Taking into account \eqref{eq2.3} and Lemma \ref{lem-6}, we get
\[
|\Omega(iy)|\asymp y^{-2m}e^{4y},\quad y\to +\infty.
\]
Now using
Lemma \ref{lem-7}, we arrive at
$f(\lambda)=q(\lambda)f_{m}(\lambda)$, where $q(\lambda)$ is an even polynomial of degree at most $2m$.
Note that the degree cannot be $2s$, $s=1,\dots,m$, since in this case \eqref{eq2.3} implies that $f\notin
L^1(\mathbb{R}_{+},\lambda^{2m-2}\,d\sigma)$. Thus,
$f(\lambda)=cf_{m}(\lambda)$, $c>0$.
\end{proof}


\section{Generalized Logan problem for Fourier transform on hyperboloid }\label{sec-Hyp}

We will use some facts of harmonic analysis on hyperboloid $\mathbb{H}^{d}$
and Lobachevskii space from \cite[Chapt.~X]{Vil78}.

Let $d\in \mathbb{N}$, $d\geq 2$, and suppose that $\mathbb{R}^{d}$ is $d$-dimensional real
Euclidean space with inner product $(x,y)=x_{1}y_{1}+\dots+x_{d}y_{d}$, and norm $|x|=\sqrt{(x,x)}$. As usual,
\[
\mathbb{S}^{d-1}=\{x\in\mathbb{R}^{d}\colon |x|=1\}
\]
is the Euclidean sphere, $\mathbb{R}^{d,1}$ is $(d+1)$-dimensional real
pseudo-Euclidean space with bilinear form $[x,y]=-x_{1}y_{1}-\dots-x_{d}y_{d}+x_{d+1}y_{d+1}$.
The upper sheet of two sheets hyperboloid is defined by
\[
\mathbb{H}^{d}=\{x\in \mathbb{R}^{d,1}\colon [x,x]=1,\,x_{d+1}>0\}
\]
and
\[
d(x,y)=\arcch{}[x,y]=\ln{}([x,y]+\sqrt{[x,y]^2-1})
\]
is the distance between $x,y\in \mathbb{H}^{d}$.

The pair $\bigl(\mathbb{H}^{d},d({\cdot},{\cdot})\bigr)$ is known as the Lobachevskii space.
Let $o=(0,\dots,0,1)\in\mathbb{H}^{d}$,  
and let $B_{r}=\{x\in \mathbb{H}^{d}\colon d(o,x)\leq r\}$ be the ball. 

In this section we will use the  Jacobi transform with parameters
$(\alpha,\beta)=(d/2-1,-1/2)$.
In particular,
\[
d\mu(t)=\Delta(t)\,dt=2^{d-1}\sh^{d-1}t\,dt,
\]
\[
d\sigma(\lambda)=s(\lambda)\,d\lambda=2^{3-2d}
\Gamma^{-2}\Bigl(\frac{d}{2}\Bigr)\Bigl|\frac{\Gamma\bigl(\frac{d-1}{2}+i\lambda\bigr)}{\Gamma(i\lambda)}\Bigr|^2d\,\lambda.
\]

For $t>0$, $\zeta\in \mathbb{S}^{d-1}$, $x=(\sh t\, \zeta,\ch t)\in \mathbb{H}^{d}$, we let
\[
d\omega(\zeta)=\frac{1}{|\mathbb{S}^{d-1}|}\,d\zeta, \quad
d\eta(x)=d\mu(t)\,d\omega(\zeta)
\]
be the Lebesgue measures on $\mathbb{S}^{d-1}$ and $\mathbb{H}^{d}$, respectively.
Note that $d\omega$ is the probability measure on the sphere, invariant under rotation group $SO(d)$
and the measure $d\eta$ is invariant under hyperbolic rotation group $SO_0(d,1)$.

For $\lambda\in \mathbb{R}_{+}=[0, \infty)$, $\xi\in \mathbb{S}^{d-1}$, $y=(\lambda, \xi)\in \mathbb{R}_{+}\times \mathbb{S}^{d-1}=:\widehat{\mathbb{H}}^{d}$, we let
\[
d\hat{\eta}(y)=d\sigma(\lambda)\,d\omega(\xi).
\]
%

Harmonic analysis in $L^{2}(\mathbb{H}^{d}, d\eta)$ and $L^{2}(\widehat{\mathbb{H}}^{d}, d\hat{\eta})$ is based on 
the direct and inverse (hyperbolic) Fourier transforms
\begin{equation*}
\mathcal{F}g(y)=\int_{\mathbb{H}^{d}}g(x)[x,\xi']^{-\frac{d-1}{2}-i\lambda}\,d\eta(x),
\end{equation*}
\begin{equation*}
\mathcal{F}^{-1}f(x)=\int_{\widehat{\mathbb{H}}^{d}}f(y)[x,\xi']^{-\frac{d-1}{2}+i\lambda}\,d\hat{\eta}(y),
\end{equation*}
where $\xi'=(\xi, 1)$, $\xi\in \mathbb{S}^{d-1}$.
We stress that  the kernels of the Fourier transforms are unbounded, which cause additional difficulties. 

If $f\in L^{2}(\mathbb{H}^{d}, d\eta)$, $g\in L^{2}(\widehat{\mathbb{H}}^{d}, d\hat{\eta})$, then
\[
\mathcal{F}g\in L^{2}(\widehat{\mathbb{H}}^{d}, d\hat{\eta}),\quad \mathcal{F}^{-1}(f)\in
L^{2}(\mathbb{H}^{d}, d\eta),
\]
and $g(x)=\mathcal{F}^{-1}(\mathcal{F}g)(x)$,
$f(y)=\mathcal{F}(\mathcal{F}^{-1}f)(y)$ in the mean-square sense. The  Plancherel formulas are written as follows:
\[
\int_{\mathbb{H}^{d}}|g(x)|^2\,d\eta(x)=\int_{\widehat{\mathbb{H}}^{d}}
|\mathcal{F}g(y)|^2\,d\hat{\eta}(y),
\]
\[
\int_{\widehat{\mathbb{H}}^{d}}|f(y)|^2\,d\hat{\eta}(y)=\int_{\mathbb{H}^{d}}
|\mathcal{F}^{-1}f(x)|^2\,d\eta(x).
\]


The Jacobi function
$\varphi_{\lambda}(t)=\varphi_{\lambda}^{(d/2-1,-1/2)}(t)$ is obtained by
averaging over the sphere of Fourier transform kernels
\[
\varphi_{\lambda}(t)=\int_{\mathbb{S}^{d-1}}[x,\xi']^{-\frac{d-1}{2}\pm i\lambda}\,d\omega(\xi),
\]
where $x=(\sh t\,\zeta,\ch t)$, $\zeta\in \mathbb{S}^{d-1}$, $\xi'=(\xi,1)$.
We note that spherical functions $g(x)=g_0(d(o,x))=g_0(t)$ and $f(y)=f_0(\lambda)$
 satisfy
\[
\mathcal{F}g(y)=\mathcal{J}g_{0}(\lambda),\quad \mathcal{F}^{-1}f(x)=\mathcal{J}^{-1}f_{0}(t).
\]

{To pose $m$-Logan problem in the case of the hyperboloid,}
let $f(y)$ be a real-valued continuous  function on $\widehat{\mathbb{H}}^{d}$, $y=(\lambda, \xi)$ and let
\[
\Lambda(f)= \Lambda(f,\widehat{\mathbb{H}}^{d})=\sup\,\{\lambda>0\colon f(y)=f(\lambda,\xi)>0,\ \xi\in\mathbb{S}^{d-1}\}
\]
{and, as above, $\Lambda_{m}(f)=\Lambda((-1)^{m-1}f)$, $m\in\mathbb{N}$.}


Consider the class $\mathcal{L}_{m}(\tau,\widehat{\mathbb{H}}^{d})$ of real-valued functions $f$
on $\widehat{\mathbb{H}}^{d}$ such that

\smallbreak
(1) $f\in L^1(\widehat{\mathbb{H}}^{d}, \lambda^{2m-2}\,d\hat{\eta}(\lambda))\cap C_b(\widehat{\mathbb{H}}^{d})$, \ $f\ne 0$, \
$\mathcal{F}^{-1}f\ge 0$, \ $\mathrm{supp}\,\mathcal{F}^{-1}f\subset B_{2\tau}$;

\smallbreak
(2) $\int_{\widehat{\mathbb{H}}^{d}}\lambda^{2k}f(y)\,d\hat{\eta}(y)=0$, \ $k=0,1,\dots,m-1$.


\begin{probE}
Find 
\[
L_m(\tau, \widehat{\mathbb{H}}^{d})=\inf \{\Lambda_{m}(f)\colon f\in \mathcal{L}_{m}(\tau,\widehat{\mathbb{H}}^{d})\}.
\]
\end{probE}

Let us show that in the generalized Logan problem on hyperboloid, one can restrict oneself
to only spherical functions depending on $\lambda$.

If a function $f\in \mathcal{L}_{m}(\tau,\widehat{\mathbb{H}}^{d})$ and
$y=(\lambda, \xi)\in \widehat{\mathbb{H}}^{d}$,
then the function
\[
f_{0}(\lambda)=\int_{\mathbb{S}^{d-1}}f(y)\,d\omega(\xi)
\]
satisfies the following properties:

\smallbreak
(1) $f_0\in L^1(\mathbb{R}_{+}, \lambda^{2m-2}d\sigma)\cap C_b(\mathbb{R}_{+}), \ f_0\ne 0, \
\mathcal{J}^{-1}f_{0}(t)\ge 0, \ \mathrm{supp}\,\mathcal{J}^{-1}f_0\subset [0,2\tau]$;

\smallbreak
(2) $\int_{0}^{\infty}\lambda^{2k}f_{0}(\lambda)\,d\sigma(\lambda)=0$, \ $k=0,1,\dots,m-1$;

\smallbreak
(3) $\Lambda_m(f_{0},\mathbb{R}_{+})=\Lambda_m(f,\widehat{\mathbb{H}}^{d})$.

\smallbreak
By Paley--Wiener theorem (see Lemma~\ref{lem-3}) $f_{0}\in
\mathcal{B}_1^{2\tau}$,
\[
f_0(\lambda)=\int_{0}^{2\tau}\mathcal{J}^{-1}f_0(t)\varphi_{\lambda}(t)\,d\mu(t)
\]
and $f_0\in\mathcal{L}_{m}(\tau,\mathbb{R}_{+})$. Hence, $L_m(\tau, \widehat{\mathbb{H}}^{d})=L_m(\tau, \mathbb{R}_{+})$
and from Theorem~\ref{thm-1} we derive the following result.

\begin{thm}
If $d,m\in\mathbb{N}$, $\tau>0$, $\lambda_1(\tau)<\dots<\lambda_m(\tau)$ are
the zeros of $\varphi_{\lambda}^{(d/2-1,-1/2)}(\tau)$, then
\[
L_m(\tau, \widehat{\mathbb{H}}^{d})=\lambda_m(\tau).
\]
The extremizer
\[
f_m(y)=
\frac{(\varphi_{\lambda}^{(d/2-1,-1/2)}(\tau))^2}{(1-\lambda^2/\lambda_1^2(\tau))\cdots(1-\lambda^2/\lambda_m^2(\tau))},
\quad y=(\lambda, \xi)\in\widehat{\mathbb{H}}^{d},
\]
is {unique} in the class of spherical functions up to multiplication by a positive constant.
\end{thm}

 \section{Number of zeros of positive definite function}\label{sec-JPD}
In  \cite{Lo83c}, it was proved that
$[0,\pi n/4\tau]$ is the minimal interval
 containing  not less $n$ zeros of functions from the class \eqref{eq1.1}.
Moreover, in this case
\[
F_{n}
(x)=\Bigl(\cos \frac{2\tau x}{n}\Bigr)^n
\]
is the unique extremal function.

Note that $x=\pi n/4\tau$ is a unique zero of $F_{n}$ on $[0,\pi n/4\tau]$ of
 multiplicity $n$.
 Moreover, the functions $F_{n}(\pi n(x-1/4\tau))$ for $n=1$ and $3$
 coincide, up to constants, with the cosine Fourier transform of $f_1$ and $f_2$ (see Introduction) on $[0,1]$.

In this section we study a similar problem for the Jacobi transform $\mathcal{J}$
with $\alpha\ge\beta\ge -1/2$, $\alpha>-1/2$. For the Bessel transform, this question was investigated in \cite{GIT19}.
We will use the approach which was developed in Section~\ref{sec-LJ}.
The key argument in the proof is based on
the properties of the polynomial $p_{m}(t)$ defined in
\eqref{eq30}.

Recall that $N_{I}(g)$ stands from the number of zeros of $g$ on interval
$I\subset \mathbb{R}_{+}$, counting multiplicity and $\lambda_m(t)$,
 $\lambda_m^{*}(t)$ are the zeros of functions $\varphi_{\lambda}(t)$ (see~\eqref{eq1.3}) and $\psi_{\lambda}'(t)$ (see \eqref{eq2.6}), respectively,
and, moreover,
$t_m(\lambda)$, $t_m^{*}(\lambda)$ are inverse functions for $\lambda_m(t)$ and
$\lambda_m^{*}(t)$.

We say that  $g\in\mathcal{L}^+_\gamma$, $\gamma>0$, if
\begin{equation}\label{eq4.1}
g(t)=\int_{0}^{\gamma}\varphi_{\lambda}(t)\,d\nu(\lambda),\quad g(0)>0,
\end{equation}
with a nonnegative bounded Stieltjes measure $d\nu$. Note that the function
$g(t)$ is analytic on $\mathbb{R}$ {but not entire}.

We set, {for $g\in \mathcal{L}^+_\gamma$},
\[
\mathrm{L}\,(g,n):=\inf{}\{L>0\colon N_{[0,L]}(g)\ge n\},\quad n\in \mathbb{N}.
\]

\begin{thm}\label{thm-2}
We have
\begin{equation}\label{eq4.2}
\inf_{g\in\mathcal{L}^+_\gamma} \mathrm{L}\,(g,n)\le \theta_{n,\gamma}=
\begin{cases}
t_{m}(\gamma),& n=2m-1,\\ t_{m}^{*}(\gamma),& n=2m.
\end{cases}
\end{equation}
Moreover, there exists a positive definite function
$G_{n}\in\mathcal{L}^+_\gamma$ such that $\mathrm{L}\,(G_{n},n)= \theta_{n,\gamma}$.
\end{thm}



\begin{proof} Put $\tau:=t_m(\gamma)$.
First, let $n=2m-1$. Consider the polynomial (see \eqref{eq30})
\[
G_{n}(t)=
\sum_{i=1}^mB_i(\tau)\varphi_{\lambda_i(\tau)}(t),\quad t\in \mathbb{R}_{+},
\]
constructed in Theorem~\ref{thm-1}. It has positive coefficients $B_i(\tau)$ and the unique zero $t=\tau$ of multiplicity
$2m-1$ on the interval $[0,\tau]$.
Hence, $G_{n}$ is of the form \eqref{eq4.1}, positive definite, and such that $t=\tau$ is a unique zero of multiplicity $2m-1$
on the interval $[0,\tau]$. Therefore,
\[
\mathrm{L}\,(G_n,2m-1)\le \tau.
\]

\smallbreak
Second, let $n=2m$, $\lambda_{i}^{*}:=\lambda_{i}^{*}(\tau)$. As in Theorem~\ref{thm-1} we define numbers
$A_i^{*}:=A_i^{*}(\tau)$ from the relation
\begin{equation}\label{eq4.3}
\sum_{i=1}^{m}\frac{A_i^{*}}{\lambda_{i}^{*\,2}-\lambda^2}=
\frac{1}{\prod_{i=1}^{m}(1-\lambda^2/\lambda_{i}^{*\,2})}.
\end{equation}
Recall that $\mathrm{sign}\,A_i^{*}=(-1)^{i-1}$. Set
\[
\Psi_i(t):=\psi_{\lambda_{i}^{*}}(t)-\psi_{\lambda_{i}^{*}}(\tau),\quad i=1,\dots,m,
\]
where $\psi_{\lambda}(t)$ is defined in \eqref{eq2.6}. In view of \eqref{eq2.7}, $\psi_{\lambda_{i}^{*}}(t)$ are eigenfunctions
and $\lambda_{i}^{*\,2}$ are eigenvalues of the following Sturm--Liouville problem on $[0,\tau]$:
\[
(w(t)u'(t))'+\lambda^2w(t)u(t)=0,\quad u'(0)=0,\quad u'(\tau)=0,
\]
where the weight $w(t)=\varphi_0^2(t)\Delta(t)$. Since $\Psi_i'(\tau)=0$, then from equation
\begin{equation}\label{eq4.4}
(w(t)\Psi_i'(t)(t))'+\lambda^2w(t)\psi_{\lambda_{i}^{*}}(t)=0
\end{equation}
it follows $\Psi_i''(\tau)=-\lambda_{i}^{*\,2}\psi_{\lambda_{i}^{*}}(\tau)$.

Let us consider the polynomial
\begin{equation}\label{eq4.5}
r_m(t)=\sum_{i=1}^{m}A_i^{*}\,\frac{\Psi_i(t)}{\Psi_i''(\tau)}=:\sum_{i=1}^{m}B_i^{*}\Psi_i(t).
\end{equation}
By \eqref{eq2.8}, $\mathrm{sign}\,\psi_{\lambda_{i}^{*}}(\tau)=(-1)^i$, hence, $B_i^{*}>0$, $r_m(0)>0$, $r_m(\tau)=0$.

Let us show that at the point $t=\tau$ polynomial $r_m(t)$ has zero of order $2m$.
As in Theorem~\ref{thm-1},
\begin{align}
r_m(t)&=c_1\sum_{i=1}^{m}(-1)^{i-1}\Delta(\lambda_1^{*\,2},\dots,\lambda_{i-1}^{*\,2},\lambda_{i+1}^{*\,2},\dots,\lambda_{m}^{*\,2})\,
\frac{\Psi_i(t)}{\Psi_i''(\tau)}\notag\\
&=c_1\begin{vmatrix}
\frac{\Psi_{1}(t)}{\Psi_{1}''(\tau)}&\dots&\frac{\Psi_{m}(t)}{\Psi_{m}''(\tau)}\\
1&\dots&1\\
\lambda_{1}^{*\,2}&\dots &\lambda_{m}^{*\,2}\\
\hdotsfor[2]{3}\\
\lambda_{1}^{*\,(2m-2)}&\dots &\lambda_{m}^{*\,(2m-2)}
\end{vmatrix},\quad c_1>0.
\label{eq4.6}
\end{align}

Show that
\begin{equation}\label{eq4.7}
\begin{vmatrix}
\frac{\Psi_{1}(t)}{\Psi_{1}''(\tau)}&\dots&\frac{\Psi_{m}(t)}{\Psi_{m}''(\tau)}\\
1&\dots&1\\
\lambda_{1}^{*\,2}&\dots &\lambda_{m}^{*\,2}\\
\hdotsfor[2]{3}\\
\lambda_{1}^{*\,(2m-2)}&\dots &\lambda_{m}^{*\,(2m-2)}
\end{vmatrix}=(-1)^{\frac{m(m-1)}{2}}
\begin{vmatrix}
\frac{\Psi_{1}(t)}{\Psi_{1}''(\tau)}&\dots&\frac{\Psi_{m}(t)}{\Psi_{m}''(\tau)}\\
1&\dots&1\\
\frac{\Psi_{1}^{(4)}(t)}{\Psi_{1}''(\tau)}&\dots&\frac{\Psi_{m}^{(4)}(t)}{\Psi_{m}''(\tau)}\\
\hdotsfor[2]{3}\\
\frac{\Psi_{1}^{(2m-2)}(t)}{\Psi_{1}''(\tau)}&\dots&\frac{\Psi_{m}^{(2m-2)}(t)}{\Psi_{m}''(\tau)}
\end{vmatrix}.
\end{equation}

Differentiating \eqref{eq4.4} and substituting $t=\tau$, we get for $s\ge 1$
\begin{eqnarray*}
\Psi_i^{(s+2)}(\tau)&=&-w'(\tau)w^{-1}(\tau)\Psi_i^{(s+1)}(\tau)-(s(w'(\tau)w^{-1}(\tau))'+\lambda_i'^2)
\Psi_i^{(s)}(\tau)
\\
&-&\sum_{j=2}^{s-1}\binom{s}{j-1}(w'(\tau)w^{-1}(\tau))^{(s+1-j)}\Psi_i^{(j)}(\tau),\quad \Psi_i(\tau)=\Psi_i'(\tau)=0.
\end{eqnarray*}
From this recurrence formula by induction we deduce that for $k=0,1,\dots$
\begin{equation}\label{eq4.8}
\begin{gathered}
\Psi_i^{(2k+2)}(\tau)=(r_0^k+r_1^k\lambda_i^{*\,2}+\cdots+r_k^k\lambda_i^{*\,2k})\Psi_i''(\tau),\\
\Psi_i^{(2k+3)}(\tau)=(p_0^k+p_1^k\lambda_i^{*\,2}+\cdots+p_k^k\lambda_i^{*\,2k})\Psi_i''(\tau),
\end{gathered}
\end{equation}
where $r_0^k,\dots,r_k^k$, $p_0^k,\dots,p_k^k$ depend from $\alpha$, $\beta$, $\tau$
and do not depend from $\lambda_i^{*}$. Moreover, $r_k^k=(-1)^k$.

From \eqref{eq4.8} it follows that for $k=1,2,\dots$
\begin{equation}\label{eq4.9}
\frac{\Psi_i^{(2k+1)}(\tau)}{\Psi_i''(\tau)}=\sum_{s=1}^{k}c_{s}^1\,
\frac{\Psi_i^{(2s)}(\tau)}{\Psi_i''(\tau)},
\end{equation}
\begin{equation}\label{eq4.10}
\frac{\Psi_i^{(2k+2)}(\tau)}{\Psi_i''(\tau)}=\sum_{s=1}^{k}c_{s}^2\,
\frac{\Psi_i^{(2s)}(\tau)}{\Psi_i''(\tau)}+(-1)^k\lambda_i'^{2k},
\end{equation}
where $c_{s}^1$, $c_{s}^2$ do not depend from $\lambda_i^{*}$.
Applying \eqref{eq4.10}, we obtain \eqref{eq4.7}
\[
\begin{vmatrix}
\frac{\Psi_{1}(t)}{\Psi_{1}''(\tau)}&\dots&\frac{\Psi_{m}(t)}{\Psi_{m}''(\tau)}\\
1&\dots&1\\
\frac{\Psi_{1}^{(4)}(t)}{\Psi_{1}''(\tau)}&\dots&\frac{\Psi_{m}^{(4)}(t)}{\Psi_{m}''(\tau)}\\
\hdotsfor[2]{3}\\
\frac{\Psi_{1}^{(2m-2)}(t)}{\Psi_{1}''(\tau)}&\dots&\frac{\Psi_{m}^{(2m-2)}(t)}{\Psi_{m}''(\tau)}
\end{vmatrix}=
\begin{vmatrix}
\frac{\Psi_{1}(t)}{\Psi_{1}''(\tau)}&\dots&\frac{\Psi_{m}(t)}{\Psi_{m}''(\tau)}\\
1&\dots&1\\
(-1)\lambda_1^{*\,2}&\dots &(-1)\lambda_{m}^{*\,2}\\
\hdotsfor[2]{3}\\
(-1)^{m-1}\lambda_1^{*\,(2m-2)}&\dots &(-)^{m-1}\lambda_{m}^{*\,(2m-2)}
\end{vmatrix}.
\]

The equalities \eqref{eq4.6} and \eqref{eq4.7} mean that
\[
r_m(\tau)=r_m''(\tau)=\cdots=r_m^{(2m-2)}(\tau)=0.
\]
According to \eqref{eq4.9},
\begin{eqnarray*}
r_m^{(2j+1)}(\tau)&=&\sum_{i=1}^{m}A_i^{*}\,\frac{\Psi_i^{(2j+1)}(\tau)}
{\Psi_i''(\tau)}=\sum_{i=1}^{m}A_i^{*}\sum_{s=1}^{j}c_s^1(\alpha)\,
\frac{\Psi_i^{(2s)}(\tau)} {\Psi_i'(\tau)}\\
&=&\sum_{s=1}^{j}c_s^1(\alpha)\sum_{i=1}^{m}A_i^{*}\,\frac{\Psi_i^{(2s)}(\tau)}
{\Psi_i'(\tau)}=\sum_{s=1}^{j}c_s^1(\alpha)r^{(2s)}(\tau)=0,\quad j=1,\dots,m-1.
\end{eqnarray*}
Since $r_m'(\tau)=0$, then at the point $t=\tau$ the polynomial $r_m(t)$ has zero of multiplicity~$2m$.

We show that it has no other zeros on the interval $[0,\tau]$. We take into account  that the system  $\{\psi_{i}(t)\}_{i=1}^m$
is a Chebyshev system on the interval $(0,\tau)$ (see Theorem~\ref{thm-5} above) and any polynomial of order $m$ on the interval
$(0,\tau)$ has at most $m-1$ zeros, counting multiplicity.

We consider the following polynomial in Chebyshev system $\{\Psi_i(t)\}_{i=1}^m$:
\begin{equation}\label{eq4.11}
r_m(t,\varepsilon)=
\begin{vmatrix}
\frac{\Psi_{1}(t)}{\Psi_{1}''(\tau)}&\dots&\frac{\Psi_{m}(t)}{\Psi_{m}''(\tau)}\\
\frac{\Psi_{1}(\tau-\varepsilon)}{\varepsilon^2\Psi_{1}''(\tau)}&\dots&\frac{\Psi_{m}(\tau-\varepsilon)}{\varepsilon^2\Psi_{m}''(\tau)}\\
\frac{\Psi_{1}(\tau-2\varepsilon)}{(2\varepsilon)^4\Psi_{1}''(\tau)}&\dots&\frac{\Psi_{m}(\tau-2\varepsilon)}{(2\varepsilon)^4\Psi_{1}''(\tau)}\\
\hdotsfor[2]{3}\\
\frac{\Psi_{1}(\tau-(m-1)\varepsilon)}{((m-1)\varepsilon)^{2m-2}\Psi_{1}''(\tau)}
&\dots&\frac{\Psi_{m}(\tau-(m-1)\varepsilon)}{((m-1)\varepsilon)^{2m-2}\Psi_{m}''(\tau)}
\end{vmatrix}.
\end{equation}
 For any $0<\varepsilon<\tau/(m-1)$, it has $m-1$
zeros at the points $t_j=\tau-j\varepsilon$, $j=1,\dots,m-1$, and has no other zeros on $(0,\tau)$.
The limit polynomial as $\varepsilon\to 0$ does not have zeros on $(0,\tau)$.

In order to calculate it, we apply  the expansions


\[
\frac{\Psi_i(\tau-j\varepsilon)}{(j\varepsilon)^{2j}\Psi_i''(\tau)}=
\sum_{s=2}^{2j-1}\frac{\Psi_i^{(s)}(\tau)}{s!\,(-j\varepsilon)^{2j-s}\Psi_i''(\tau)}+
\frac{\Psi_i^{(2j)}(\tau)+o(1)}{(2j)!\,\Psi_i''(\tau)},\quad j=1,\dots,m-1,
\]
formulas \eqref{eq4.9} and \eqref{eq4.10}, and we subtract  successively in the determinant \eqref{eq4.11}
from the subsequent rows the previous ones to obtain
\[
r_m(t,\varepsilon)=\frac{1}{\prod_{j=1}^m(2j)!}
\begin{vmatrix}
\frac{\Psi_{1}(t)}{\Psi_{1}''(\tau)}&\dots&\frac{\Psi_{m}(t)}{\Psi_{m}''(\tau)}\\
1+o(1)&\dots&
1+o(1)\\
\frac{\Psi_{1}^{(4)}(\tau)+o(1)}{\Psi_{1}''(\tau)}
&\dots&\frac{\Psi_{m}^{(4)}(\tau)+o(1)}{\Psi_{m}''(\tau)}\\
\hdotsfor[2]{3}\\
\frac{\Psi_{1}^{(2m-2)}(\tau)+o(1)}{\Psi_{1}''(\tau)}
&\dots&\frac{\Psi_{m}^{(2m-2)}(\tau)+o(1)}{\Psi_{m}''(\tau)}
\end{vmatrix}.
\]
From here and  \eqref{eq4.6}, \eqref{eq4.7} it follows  that
\[
\lim\limits_{\varepsilon\to 0}r_m(t,\varepsilon)=c_2r_m(t),\quad c_2>0.
\]
Hence, the polynomial $r_m(t)$ is positive on the interval $[0,\tau)$.

The polynomial $r_m(t,\varepsilon)$ vanishes at $m$ points, including $\tau$, and therefore
its derivative $r_m'(t,\varepsilon)$ has $m-1$ zeros between $\tau-(m-1)\varepsilon$ and $\tau$.
Since the system $\{\Psi_i'(t)\}_{i=1}^{m}$ is  the Chebyshev system on $(0,\tau)$ (see Theorem~\ref{thm-5} below),
then $r_m'(t,\varepsilon)$ does not have zeros on $(0,\tau)$. Hence, for $\varepsilon\to 0$ we derive that $r_m'(t)$
does not have zeros on $(0,\tau)$. Since $r_{m}(0)>0$ and $r_{m}(\tau)=0$, then $r_{m}'(t)<0$ on $(0,\tau)$.
Thus, the polynomial $r_{m}(t)$ decreases on the interval
$[0,\tau]$.

Since $\Psi_i''(\tau)=-\lambda_{i}^{*\,2}\psi_{\lambda_{i}^{*}}(\tau)$, polynomial \eqref{eq4.5} can be written as
\[
r_m(t)=\sum_{i=1}^m\frac{A_i^{*}(\tau)}{\lambda_{i}^{*\,2}}+\sum_{i=1}^mB_i^{*}(\tau)\psi_{\lambda_i^{*}(\tau)}(t).
\]
Setting $\lambda=0$ in \eqref{eq4.3} we obtain
$
\sum_{i=1}^m\frac{A_i^{*}(\tau)}{\lambda_{i}^{*\,2}}=1,
$ 
 therefore
\[
r_m(t)=1+\sum_{i=1}^mB_i^{*}(\tau)\psi_{\lambda_i^{*}(\tau)}(t).
\]
This  polynomial has positive coefficients and the unique zero $t=\tau$ of multiplicity $2m$ on the interval $[0,\tau]$. Since $\psi_{\lambda}(t)=\varphi_{\lambda}(t)/\varphi_0(t)$ and $\varphi_0(t)>0$, the function
\[
G_{n}(t)=\varphi_0(t)+\sum_{i=1}^mB_i^{*}(t_m^{*}(\gamma))\varphi_{\lambda_i^{*}(t_m^{*}(\gamma))}(t)
\]
is of the form \eqref{eq4.1}, positive definite, and such that $t=t_m^{*}(\gamma)$ is a unique zero of multiplicity~$2m$ on the interval $[0,t_m(\gamma)]$. Hence,
\[
\mathrm{L}\,(G_n,2m)\le t_m^{*}(\gamma).
\]
\end{proof}


\begin{rem}
From the proof of Theorem~\ref{thm-2} it follows that inequality \eqref{eq4.2} is also valid
for functions  represented by
\[
g(t)=\int_{0}^{\gamma}\psi_{\lambda}(t)\,d\nu(\lambda),\quad g(0)>0,
\]
with a nonnegative bounded Stieltjes measure $d\nu$.
\end{rem}

\end{document}